\documentclass[11pt]{article}
\pdfoutput=1
\usepackage{amsbsy,amsmath,amsthm,amssymb,graphicx,verbatim}
\usepackage{setspace, float}
\usepackage[longnamesfirst,sort]{natbib}

\usepackage{multirow}
\usepackage{subcaption}
\usepackage{hhline}

\usepackage{geometry}
\newtheorem{theorem}{Theorem}
\newtheorem{corollary}{Corollary}
\newtheorem{lemma}{Lemma}
\newtheorem{proposition}{Proposition}
\newtheorem{assumption}{Assumption}
\theoremstyle{remark}
\newtheorem{remark}{Remark}

\newfont{\msbm}{msbm10 at 11pt}

\newcommand{\Var}[0]{\text{Var}}

\begin{document}
\onehalfspacing

\title{Weighted batch means estimators in Markov chain Monte Carlo}

\author{Ying Liu \\ Department of Statistics \\ University of California, Riverside \\ {\tt yliu055@email.ucr.edu}\and James M. Flegal \\ Department of Statistics \\ University of California, Riverside \\ {\tt jflegal@ucr.edu}   }  

\date{\today}
\maketitle
\begin{abstract}
This paper proposes a family of weighted batch means variance estimators, which are computationally efficient and can be conveniently applied in practice.  The focus is on Markov chain Monte Carlo simulations and estimation of the asymptotic covariance matrix in the Markov chain central limit theorem, where conditions ensuring strong consistency are provided.  Finite sample performance is evaluated through auto-regressive, Bayesian spatial-temporal, and Bayesian logistic regression examples, where the new estimators show significant computational gains with a minor sacrifice in variance compared with existing methods.
\end{abstract}

\section{Introduction}

Markov chain Monte Carlo (MCMC) methods are widely used to approximate expectations with respect to a target distribution, see e.g. \cite{liu:2001} and \cite{robe:case:2004}.  In short, an MCMC simulation generates a dependent sample from the target distribution and then uses ergodic averages to estimate a vector of expectations.  Variability of the ergodic averages is of interest because it reflects the quality of estimation and can be used to construct confidence intervals or confidence regions \citep[see e.g.][]{fleg:hara:jone:2008, fleg:jone:2011, geye:1992, vats:fleg:jone:2015output, jone:hobe:2001}.  Estimating variability is akin to estimation of the asymptotic covariance matrix in a multivariate Markov chain central limit theorem (CLT).

Let $F$ be a probability distribution with support $\mathsf{X} \in \mathbb{R}^{d}$ and $g: \mathsf{X}\rightarrow \mathbb{R}^{p}$ be a $F$-integrable function. We are interested in estimating the $p$-dimensional vector
\[
\theta = \int_{\mathsf{X}} g(x) dF.
\]
Let $X=\{X_{t}, t\geq 1\}$ be a Harris ergodic Markov chain with invariant distribution $F$.  Then if $Y_{t}=g(X_{t})$ for $t \ge 1$, $\bar{Y}_{n} = \dfrac{1}{n}\sum_{t=1}^{n} Y_{t}\rightarrow\theta$ w.p.\ 1 as $n\rightarrow\infty$.  The sampling distribution for $\bar{Y}_{n} - \theta$ is available via a Markov chain CLT if there exists a positive definite symmetric matrix $\Sigma$ such that 
\begin{equation} \label{eq:clt}
\sqrt{n}(\bar{Y}_{n}-\theta)\xrightarrow{d}N_{p}(0, \Sigma) \text{ as } n\rightarrow\infty,
\end{equation}
where
\[
\Sigma=\text{Var}_{F}(Y_{1})+\sum_{k=1}^{\infty}[\text{Cov}_{F}(Y_{1}, Y_{1+k})+\text{Cov}_{F}(Y_{1} , Y_{1+k})^T].
\]
Provided an estimator of $\Sigma$ is available, say $\hat{\Sigma}_n$, one can access variability of the estimator $\bar{Y}_{n}$ by constructing a $p$-dimensional ellipsoid.  Further, \cite{vats:fleg:jone:2015output} propose terminating the simulation when the ellipsoid volume is sufficiently small, which is asymptotically equivalent to stopping when a multivariate effective sample size is large enough.  One of their necessary conditions is that $\hat{\Sigma}_n$ is a strongly consistent estimator of $\Sigma$.

Outside of recent work of \cite{chan2017automatic}, \cite{dai:jone:2017}, \cite{vats:fleg:jone:2015output}, and \cite{vats:fleg:jone:2018}, estimating the covariance matrix is rarely done in MCMC.  Instead most practitioners focus on univariate techniques to estimate only the diagonal components.  An incomplete list of univariate estimators includes batch means (BM) and overlapping BM \citep{jone:hara:caff:neat:2006, fleg:jone:2010, meke:schm:1984}, spectral variance (SV) methods including flat top estimators \citep{ande:1994, poli:roma:1996, poli:roma:1995}, initial sequence estimators \citep{geye:1992}, recursive estimators of time-average variances \citep{wu2009recursive, yau2016new}, and regenerative simulation \citep{mykl:tier:yu:1995, hobe:jone:pres:rose:2002, seil:1982}.  Many of these univariate techniques can be extended to the multivariate setting, but practical challenges increase as the dimension increases.  

Within the MCMC literature, the most common approach is univariate BM since it is fast and simple to calculate.  Speedy calculations are especially helpful in conjunction with sequential stopping rules where multiple variances or a covariance matrix would be calculated each time a stopping criteria is checked \citep[see e.g.][]{fleg:hara:jone:2008, gong:fleg:2016}.  Unfortunately, \cite{fleg:jone:2010} and \cite{vats:fleg:jone:2015output} illustrate BM methods tend to underestimate confidence region volumes unless the number of Markov chain iterations is extremely large.  Practitioners familiar with the time-series literature may argue for more complex SV estimators using Tukey-Hanning or flat top lag windows.  Flat top windows are especially appealing since they tend to reduce bias leading to more accurate confidence region volumes.  Despite the popularity in fields where sample sizes are moderate, multivariate SV methods are challenging to use in MCMC since they require substantial computational effort for large sample sizes (see Section~\ref{sec:comptime}).

This paper introduces weighted BM variances estimators that are especially convenient in MCMC but are applicable in other fields such as time-series and nonparametric analysis.  The proposed estimators incorporate the same flexible lag windows of SV estimators while reducing computation time.  For example, we later show a weighted BM estimator is approximately 60 times faster for a $30 \times 30$ covariance matrix with $5e5$ iterations.  Moreover, the speed up increases as dimension or iteration increases.  

The cost one pays for computational efficiency is an increase in relative efficiency.  Specifically, we show the variance is 1.875 higher for a flat top lag window using weighted BM versus a traditional SV estimator.  Our result is similar to \cite{fleg:jone:2010} who show the variance of the BM estimator is 1.5 times higher than that of the overlapping BM estimator.

In addition to calculating relative efficiency, we prove strong consistency of weighted BM estimators.  Strong consistency is important since it is required for asymptotic validity of sequential stopping rules, see e.g.\ \cite{fleg:gong:2015}, \cite{glyn:whit:1992}, \cite{jone:hara:caff:neat:2006}, and \cite{vats:fleg:jone:2015output}.  In short, asymptotic validity implies the simulation terminates with probability one and ensures the final confidence regions have the right coverage probability.

The performance of weighted BM estimators is illustrated in univariate and multivariate auto-regressive models. These finite sample simulations show weighted BM estimators converge to the true known value and that flat top lag windows enjoy significant bias reduction.  As dimension or chain length increases, calculation of weighted BM estimators save significant time compared with SV estimators.  Our simulations also illustrate an increase in the variance relative to SV estimators, which depends lag window choice.  

We also consider a Bayesian spatial-temporal model applied to temperature data collected from ten nearby weather station in the year 2010. In this example, we estimate the covariance matrix associated with a vector of 185 parameters and again illustrate the improved computational efficiency of weighted BM estimators.  Our final example considers a Bayesian logistic regression model that illustrates weighted BM estimators with a flat top window provide more accurate coverage probabilities of multivariate confidence regions.

The rest of the paper is organized as follows.  Section~\ref{sec:covariance} summarizes current multivariate estimators of $\Sigma$.  Section~\ref{sec:weightedBM} proposes weighted BM estimators, establishes conditions that ensure strong consistency, and calculates the variance when using a Bartlett flat top lag window.  Section~\ref{sec:weightedBM} also investigates how chain length $n$ and dimension $p$ impact computation times for weighted BM, SV, and recursive estimators.  Section~\ref{sec:examples} demonstrates the finite sample properties of weighted BM estimators via four examples.  We conclude with a discussion in Section~\ref{sec:discussion}.  All proofs are relegated to the Appendix.

\section{Covariance matrix estimation} \label{sec:covariance}

Estimating $\Sigma$ is rarely done in MCMC output analysis.  Instead, most researchers ignore the cross-correlation and only estimate the diagonal entries of $\Sigma$.  Computationally efficient BM methods are usually preferred, but such methods can lead to lower than expected coverage probabilities.  In this section, we provide formal definitions for existing estimators of $\Sigma$ and provide some motivation for our proposed weighted BM estimators.  When $p=1$ these estimators reduce to the usual univariate estimators. 

First consider BM estimators where $a=a_n$ is the number of batches, $b=b_n$ is the batch size, and $n=ab$.  (Note $a$ and $b$ can depend on $n$, but we suppress this dependency to simplify notation.)  For $l=0,1,...,a-1$, denote the mean vector for batch $l$ as $\bar{Y}_{l}(b)=b^{-1}\sum_{t=1}^{b}Y_{lb+t}$.  Then the sample variance of batch means scaled up properly is used to estimate $\Sigma$, i.e.\
\begin{equation} \label{eq:bm}
\hat{\Sigma}_{bm}=\dfrac{b}{a-1}\sum_{l=0}^{a-1}(\bar{Y}_{l}(b)-\bar{Y}_{n})(\bar{Y}_{l}(b)-\bar{Y}_{n})^T.
\end{equation}
Alternatively, overlapping BM use $n-b+1$ overlapping batches of length $b$ denoted $\dot{Y}_{l}(b)=b^{-1}\sum_{t=1}^{b}Y_{l+t}$ for $l = 0, \dots, n-b$.  Then the overlapping BM estimator is given by
\begin{equation} \label{eq:olbm}
\hat{\Sigma}_{obm}=\dfrac{nb}{(n-b)(n-b+1)}\sum_{l=0}^{n-b}(\dot{Y}_{l}(b)-\bar{Y}_{n})(\dot{Y}_{l}(b)-\bar{Y}_{n})^T.
\end{equation}
Computing overlapping BM is significantly slower than BM given the increased quantity of batches.  

SV methods can also be used to estimate $\Sigma$.  First consider estimating the lag $k$ autocovariance denoted by $\Gamma(k) = \text{E}_{F} \left( Y_{t} - \theta \right) \left( Y_{t+k} - \theta \right)^{T}$ with 
\[
\hat{\Gamma}(k)=\dfrac{1}{n}\sum_{t=1}^{n-k} \left( Y_{t}-\bar{Y}_{n} \right) \left( Y_{t+k}-\bar{Y}_{n} \right)^{T}.
\]
Then the SV estimator of $\Sigma$ truncates and downweights the summed lag $k$ autocovariances.  That is, 
\[
\hat{\Sigma}_{sv}=\hat{\Gamma}(0)+\sum_{k=1}^{b}w_n(k)[\hat{\Gamma}(k)+\hat{\Gamma}(k)^T],
\]
where $b$ is the truncation point and $w_{n}(\cdot)$ is the lag window.  

We assume the lag window $w_{n}(\cdot)$ is an even function defined on $\mathbb{Z}$ such that (i) $|w_{n}(k)| \leq 1$ for all $n$ and $k$, (ii) $w_{n}(0) = 1$ for all $n$, and (iii) $w_{n}(k) = 0$ for all $|k|\geq b$.  Most commonly used lag windows satisfy this assumption, which is necessary for our proof of strong consistency.  Our discussion and simulations focus on the Bartlett, Tukey-Hanning, and Bartlett flat top lag windows defined as 
\begin{align}
w_{n}(k) & = \left( 1-|k|/b \right) I \left( |k| \le b \right) , \label{eq:bartlett} \\
w_{n}(k) & = \left( (1+\text{cos}(\pi |k|/b))/2 \right)  I \left( |k| \le b \right) \text{, and} \label{eq:th} \\
w_{n}(k) & = I \left( |k| \le b/2 \right) + \left( 2(1-|k|/b) \right)  I \left( b/2 < |k| \le b \right) , \label{eq:ft}
\end{align}
respectively (see Figure \ref{fig:windows}).  An interested reader is directed to \cite{ande:1994} for more on lag windows.
\begin{figure}[h]
    \centering
    \includegraphics[width=0.6\textwidth]{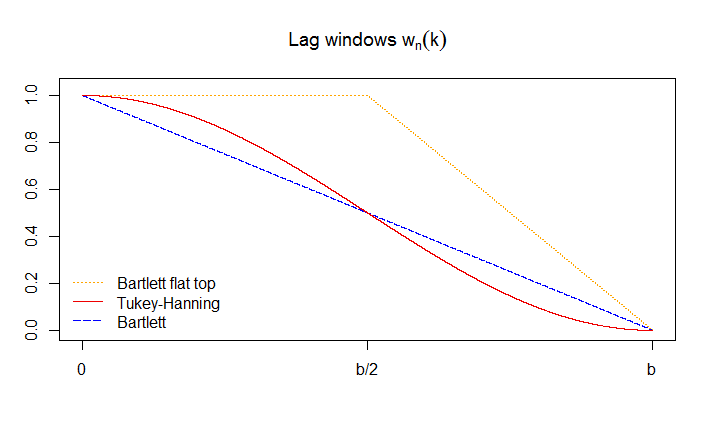}
    \caption{Plot of Bartlett, Tukey-Hanning, and Bartlett flat top lag windows.}
    \label{fig:windows}
\end{figure}

It is well known the overlapping BM estimator at \eqref{eq:olbm} is asymptotically equal to the SV estimator with a Bartlett lag window apart from some end effects \citep[see e.g.][]{welc:1987, meke:schm:1984}.  Notice in Figure \ref{fig:windows} that the Tukey-Hanning lag window slightly reduces downweighting of small lag terms compared to the Bartlett lag window in an effort to reduce bias.  \cite{poli:roma:1995,poli:roma:1996} expanded on this idea when introducing flat top lag windows that modify existing windows by letting $w_{n}(k)=1$ for $k$ near $0$.  Their work demonstrates SV estimators with flat top lag windows enjoy significant bias reduction while maintaining comparable variance.  \cite{poli:roma:1999} later illustrate the superiority of flat top lag windows in nonparametric estimation of multivariate density function.

We only consider the flat top window function constructed from the Bartlett window with $w_{n}(k)=1$ for $|k| \le b/2$ as at \eqref{eq:ft}.  For this setting, \cite{poli:roma:1995,poli:roma:1996} show the resulting SV estimator is equivalent to the difference of two Bartlett SV estimators.  Specifically, if $w_{n}(k)$ is the flat top window then
\begin{equation}\label{eq:ft=bt}
\hat{\Sigma}_{sv}=\hat{\Gamma}(0)+\sum_{k=1}^{b} w_{n}(k) [\hat{\Gamma}(k)+\hat{\Gamma}(k)^{T}]= 2\hat{\Sigma}^{(1)}-\hat{\Sigma}^{(2)},
\end{equation}
where $\hat{\Sigma}^{(1)}$ and $\hat{\Sigma}^{(2)}$ denote Bartlett SV estimators with bandwidths $b$ and $b/2$, respectively.

In the next section, we construct weighted BM estimators that inherit desired properties from lag window functions but are computationally efficient due to a nonoverlapping structure.

\section{Weighted BM estimators} \label{sec:weightedBM}

Consider first an alternative representation of the SV estimator that is akin the overlapping BM estimator.  Similar estimators have been previously studied by \cite{dame:1987, dame:1991} and \cite{fleg:jone:2010}.  To this end, define $\Delta_{1}w_{n}(k)=w_{n}(k-1)-w_{n}(k)$ and $\Delta_{2}w_{n}(k)=w_{n}(k-1)-2w_{n}(k)+w_{n}(k+1)$.  Then recall $\dot{Y}_{l}(k)=k^{-1}\sum_{t=1}^{k}Y_{l+t}$ for $l=0,..., n-k$ and consider the estimator
\[
\dot{\Sigma}=\dfrac{1}{n}\sum_{k=1}^{b}\sum_{l=0}^{n-k}k^{2}\Delta_{2}w_{n}(k)(\dot{Y}_{l}(k)-\bar{Y})(\dot{Y}_{l}(k)-\bar{Y})^{T}.
\]
If $d=\hat{\Sigma}_{sv}$-$\dot{\Sigma}$, \cite{liu:fleg:2018} show that $d\rightarrow 0$ with probability 1 as $n \to \infty$, hence the estimators are asymptotically equivalent.  Starting with $\dot{\Sigma}$, it is possible to reduce number of batches and computing time by only including non-overlapping batches.  First define the more general batch mean vector as $\bar{Y}_{l}(k)=k^{-1}\sum_{t=1}^{k}Y_{lk+t}$ for $l=0,1,...,a_{k}-1$ and $k=1,2,...,b$ where $a_{k}=\lfloor (n/k)\rfloor$.  Then the weighted BM estimator is
\begin{equation} \label{eq:weighted}
\hat{\Sigma}_{w}=\sum_{k=1}^{b}\dfrac{1}{a_{k}-1}\sum_{l=0}^{a_{k}-1}k^{2}\Delta_{2}w_{n}(k)(\bar{Y}_{l}(k)-\bar{Y})(\bar{Y}_{l}(k)-\bar{Y})^{T}.
\end{equation}

The estimator $\hat{\Sigma}_{w}$ is not necessarily computationally efficient.  However, if the lag window is such that $\Delta_{2}w_{n}(k)=0$ for certain $k$ values then the first summation can be simplified.  For the Bartlett lag window at \eqref{eq:bartlett} $\Delta_{2}w_{n}(k)= 0$ for $k=1,2,...,(b-1)$ and $\Delta_{2}w_{n}(b)=1/b$.  Hence, $\hat{\Sigma}_{w}$ at \eqref{eq:weighted} reduces to the BM estimator at \eqref{eq:bm}.

We suggest using the Bartlett flat top lag window at \eqref{eq:ft} in an effort to reduce bias.  In this case, it is easy to show $\Delta_{2}w_{n}(b/2)=-2/b,\ \Delta_{2}w_{n}(b)=2/b$, and $\Delta_{2}w_{n}(k)=0$ for all other $k$ values.  Hence, the first summation in \eqref{eq:weighted} contains two terms which is extremely computationally friendly.  For this lag window, Sections~\ref{sec:comptime} and~\ref{sec:examples} illustrate computational and bias advantages, respectively.  Since the expression of $\Delta_{2}w_{n}(\cdot)$ is similar to that of a second derivative of $w_n(\cdot)$, other piecewise linear functions would also be computationally efficient.  

\subsection{Strong consistency}
This section establishes necessary conditions for strong consistency of $\hat{\Sigma}_{w}$ for estimating $\Sigma$.  Denote the Euclidean norm by $\lVert\cdot\rVert$ and let $ \{B(t), t\geq 0\}$ be a $p$-dimensional multivariate Brownian motion.  Then the primary assumption is that of a strong invariance principle.  
\begin{assumption} \label{ass:sip}
There exists a $p\times p$ lower triangular matrix $L$, a nonnegative increasing function $\psi$ on the positive integers, a finite random variable $D$, and a sufficiently rich probability space $\Omega$ such that for almost all $\omega\in \Omega$ and for all $n>n_{0}$,
\begin{equation}\label{eq:sip}
\left\lVert\sum_{t=1}^{n} Y_{t} - n \theta - LB(n)\right\rVert<D(\omega)\psi(n)\ \ \ \ \ \text{w.p.}\ 1.
\end{equation}
\end{assumption}
Our results hold as long as Assumption~\ref{ass:sip} holds.  This includes independent processes, Markov chains, Martingale sequences, renewal processes and strong mixing processes.  An interested reader is directed to \cite{vats:fleg:jone:2015output} and the references therein.

For commonly used Markov chains in MCMC settings, \cite{vats:fleg:jone:2018} show Assumption~\ref{ass:sip} holds using results from \cite{kuel:phil:1980}.  Specifically we require polynomial ergodicity, which is weaker than geometric or uniform ergodicity \citep[see e.g.][]{meyn:twee:2009}.
\begin{corollary} \label{cor:poly} \citep[Corollary 4][]{vats:fleg:jone:2018} 
Suppose $E_F \left| Y_1 \right|^{2 + \delta} < \infty$ for some $\delta > 0$. Let $X$ be an $F$-invariant polynomially ergodic Markov chain of order $m > (1 + \epsilon_1)(1+2/\delta)$ for some $\epsilon_1 > 0$.  Then for any initial distribution, \eqref{eq:sip} holds with $\psi(n) = n^{1/2 - \lambda}$ for some $\lambda >0$.  
\end{corollary}

\begin{remark}
\cite{kuel:phil:1980} show $\lambda$ only depends on $p$, $\epsilon_1$, and $\delta$, but quantifying this relationship is an open problem.  \cite{dame:1991} notes that $\lambda$ is closer to 0 for slow mixing (heavily correlated) processes and closer to 1/2 for fast mixing processes.
\end{remark}

\begin{remark}
Under stronger assumptions of geometric ergodicity, a one step minorization condition, and $p=1$, \cite{jone:hara:caff:neat:2006} and \cite{bedn:latu:2007} provide an exact relationship between $\lambda$ and the convergence rate of the chain \citep[see Lemma 3 of][]{fleg:jone:2010}.  Establishing a similar result for $p>1$ is a direction of ongoing research.
\end{remark}

The weighted BM estimator can only be consistent if the batch size increases with $n$ leading to the following additional assumption.
\begin{assumption} \label{ass:batch}
The batch size $b$ is an integer sequence such that $b\rightarrow\infty$ and $n/b\rightarrow\infty$ as $n\rightarrow\infty$, where $b$ and $n/b$ are both monotonically nondecreasing.
\end{assumption}

In Theorem~\ref{thm:consistency} we show strong consistency of $\hat{\Sigma}_{w}$. The proof is given in Appendix~\ref{proof:theo1}.
\begin{theorem} \label{thm:consistency}
Suppose the conditions of Corollary~\ref{cor:poly} hold, Assumption~\ref{ass:batch} holds, and there exists a constant $c\geq\ 1$ such that $\sum_{n}(b/n)^{c}<\infty$. If
\begin{equation} \label{eq:cond1}
\sum_{k=1}^{b}k\Delta_{2}w_{n}(k)=1,
\end{equation}
\begin{equation} \label{eq:cond2}
b n^{1 - 2\lambda} \log n\left(\sum_{k=1}^{b}|\Delta_{2}w_{n}(k)|\right)^{2}\rightarrow 0 \text{ as } n\rightarrow\infty, \text{ and }
\end{equation} 
\begin{equation} \label{eq:cond3}
n^{1 - 2\lambda} \sum_{k=1}^{b}|\Delta_{2}w_{n}(k)|\rightarrow 0 \text{ as } n\rightarrow\infty,
\end{equation}
then with probability 1, $\hat{\Sigma}_{w}\rightarrow \Sigma$ as $n\rightarrow\infty$.
\end{theorem}

\begin{remark}
\cite{fleg:jone:2010} and \cite{vats:fleg:jone:2018} include conditions at \eqref{eq:cond2} and \eqref{eq:cond3} to obtain strong consistency of univariate and multivariate SV estimators, respectively.  Lemma 1 of \cite{vats:fleg:jone:2018} is especially useful in checking these.
\end{remark}

We now consider if some common lag windows satisfy \eqref{eq:cond1}, \eqref{eq:cond2}, and \eqref{eq:cond3}.

\textit{Simple Truncation}: $w_n(k) = I(|k| < b)$. Since $\Delta_2 w_n(b)= 1$, condition \eqref{eq:cond3} is not satisfied.

\textit{Tukey-Hanning}: $w_n(k) \left( (1+\text{cos}(\pi |k|/b))/2 \right)  I( |k| \le b)$.  Appendix \ref{app:TH} provides a calculation to ensure \eqref{eq:cond1} holds.  \cite{vats:fleg:jone:2018} show for the more general Blackman-Tukey window \eqref{eq:cond2} and \eqref{eq:cond3} hold if $b^{-1} n^{1 - 2\lambda} \log n \to 0$ as $n\rightarrow\infty$ using their Lemma 1.

\textit{Parzen}: $w_n(k) = \left[1 - |k|^q/{b^q} \right] I(|k| \le b)$ for $q \in \mathbb{Z}^+$. A method of differences calculation shows \eqref{eq:cond1} holds.  \cite{vats:fleg:jone:2018} again show \eqref{eq:cond2} and \eqref{eq:cond3} hold if $b^{-1} n^{1 - 2\lambda} \log n \to 0$ as $n\rightarrow\infty$.  When $q = 1$ this is the \textit{Bartlett} window at \eqref{eq:bartlett} and $\hat{\Sigma}_{w}$ equals $\hat{\Sigma}_{bm}$ defined at \eqref{eq:bm}.  Hence Theorem~\ref{thm:consistency} provides an alternative proof of strong consistency under the same conditions as \cite{vats:fleg:jone:2015output}.
\begin{theorem} \citep[Theorem 2][]{vats:fleg:jone:2015output}
\label{thm:mbm}
Suppose the conditions of Corollary~\ref{cor:poly} hold, Assumption~\ref{ass:batch} holds, and there exists a constant $c\geq\ 1$ such that $\sum_{n}(b/n)^{c}<\infty$. If $b^{-1} n^{1 - 2\lambda} \log n \to 0$ as $n \to \infty$, then with probability 1, $\hat{\Sigma}_{bm}\rightarrow \Sigma$ as $n\rightarrow\infty$.
\end{theorem}

\textit{Scale-parameter modified Bartlett}: $w_n(k) = \left[1 -\eta |k|/{b} \right] I(|k| < b)$ where $\eta$ is a positive constant not equal to 1.  \cite{vats:fleg:jone:2018} show $\sum_{k=1}^{b} |\Delta_2 w_n(k)|$ does not converge to 0, hence \eqref{eq:cond3} is not satisfied.

\textit{Bartlett flat top}: $w_n(k) = I \left( |k| \le b/2 \right) + \left( 2(1-|k|/b) \right)  I \left( b/2 < |k| \le b \right)$.  Condition \eqref{eq:cond1} is satisfied since
\[
\sum_{k=1}^{b}\Delta_{2}w_{n}(k)=-\dfrac{2}{b}\cdot\dfrac{b}{2}+\dfrac{2}{b}\cdot b=1.
\]
However, \eqref{eq:cond3} does not hold since $\Delta_{2}w_{n}(b/2)=-2/b$ and $\Delta_{2}w_{n}(b)=2/b$.  We can still ensure strong consistency since the estimator can be expressed as the difference between two BM estimators similar to \eqref{eq:ft=bt}.  Specifically,
\begin{equation}\label{eq:WBMft=BM}
\hat{\Sigma}_{w}= 2\hat{\Sigma}_{bm}-\hat{\Sigma}_{bm}^{(2)},
\end{equation}
where $\hat{\Sigma}_{bm}$ is defined at \eqref{eq:bm} (batch size $b$) and $\hat{\Sigma}_{bm}^{(2)}$ defines a BM estimator with batch size $b/2$.  With \eqref{eq:WBMft=BM}, strong consistency follows from Theorem~\ref{thm:mbm}.
\begin{corollary}
Suppose the conditions of Theorem~\ref{thm:mbm} hold and $w_{n}(k)$ is the flat top lag window at \eqref{eq:ft}, then with probability 1, $\hat{\Sigma}_{w}\rightarrow \Sigma$ as $n\rightarrow\infty$.
\end{corollary}

A common choice is setting $b = \lfloor n^{\nu} \rfloor$ for $ 0 < \nu < 1 $.  In this case, $\nu > 1 - 2\lambda$ ensures $b^{-1} n^{1 - 2\lambda} \log n \to 0$ as $n \to \infty$.  Finite sample performance naturally depends on the choice of $\nu$.  \cite{fleg:jone:2010} and \cite{liu:fleg:2018} minimize the asymptotic mean-squared error and conclude the optimal truncation point is proportional to $\lfloor n^{1/3} \rfloor$. 


\subsection{Increase in variance}
Since weighted BM variance estimators are based only on the nonoverlapping batches, a variance inflation is expected relative to SV estimators.  Here we focus on estimating the diagonal entries of $\Sigma$ but the off-diagonal entries behave in a similar manner. 

Suppose $i \in \{ 1, \dots, p \}$ then denote estimators of the $i$th diagonal element of $\Sigma$ based on BM and weighted BM with a Bartlett flat top lag window as $\hat{\sigma}_{bm}^2$ and $\hat{\sigma}_{w}^2$, respectively.  Further, denote SV estimators with Bartlett and Bartlett flat top lag windows as $\hat{\sigma}_{b}^2$ and $\hat{\sigma}_{f}^2$, respectively.  Results in \cite{fleg:jone:2010} imply as $n\rightarrow\infty$
\[
\text{Var}[\hat{\sigma}^{2}_{bm}]/\text{Var}[\hat{\sigma}^{2}_{b}]=1.5
\]
since the overlapping BM estimator is asymptotically equivalent to $\hat{\sigma}_{b}^2$.  The variance of $\hat{\sigma}_{b}^2$ has also been studied by \cite{lahi:1999} and \cite{poli:whit:2004}.  

The following result establishes the variance ratio between weighted BM and SV estimators with a Bartlett flat top lag window at \eqref{eq:ft}. The proof is given in  Appendix~\ref{proof:var}.
\begin{theorem} \label{thm:var}
Suppose the conditions of Corollary~\ref{cor:poly} hold, Assumption \ref{ass:batch} holds, $ED^4<\infty$ in \eqref{eq:sip}, and $E_F Y_1^4<\infty$. If as $n\rightarrow\infty$, $n^{1 - 2\lambda} b^{-1} \log n \rightarrow 0$, then
\[
\text{Var}[\hat{\sigma}^{2}_{w}]/\text{Var}[\hat{\sigma}^{2}_{f}]=1.875.
\]
\end{theorem}

\begin{remark} 
For the Tukey-Hanning lag window, $\Delta_{2}w_{n}(k) \ne 0$ for all $k$ and there is no obvious simplification in the definition of $\hat{\Sigma}_{w}$ at \eqref{eq:weighted}.  Hence a variance ratio expression is challenging to obtain.  (This difficulty persists for other lag windows where a simplification in the definition of $\hat{\Sigma}_{w}$ is unavailable.)  Alternatively, this ratio can be approximated via simulation as we illustrate in Section~\ref{sec:examples}.
\end{remark}

\begin{remark}
Results in Appendix~\ref{proof:var} combined with Theorem 4 of \cite{fleg:jone:2010} yield
\[
\text{Var}[\hat{\sigma}^{2}_{f}]/\text{Var}[\hat{\sigma}^{2}_{b}]=2.
\]
\cite{poli:roma:1996, poli:roma:1995} mention such a variance increase for flat-top estimators, but go on to argue it is offset by lower bias. 
\end{remark}

\subsection{Computational time} \label{sec:comptime}

This section investigates how chain length $n$ and dimension $p$ affect computation time for weighted BM, SV, and recursive estimators.  Calculations were completed on a 2016 MacBook (1.2 GHz Intel Core m5) and coded exclusively in \texttt{R} to ensure fairness.  \cite{chan2017automatic} provide the \texttt{R}-package \texttt{rTACM} (version 3.1), which we utilize to compute recursive estimators.

As an example, consider the $p$-dimensional vector autoregressive process of order 1 (VAR(1)) 
\[
X_t=\Phi X_{t-1}+\epsilon_t,
\]
for $t=1,2,\dots$ where $X_t\in \mathbb{R}^p$, $\epsilon_t$ are i.i.d.\ $N_p(0,I_p)$ and $\Phi$ is a $p\times p$ matrix.  When the largest eigenvalue of $\Phi$ in absolute value is less than 1 the Markov chain is geometrically ergodic \citep{tjostheim1990non}.  Further, the invariant distribution is $N_{p}(0, V)$ where $vec(V)=(I_{p^2}-\Phi\otimes\Phi)^{-1}vec(I_p)$ and $\otimes$ denotes the Kronecker product. Consider approximating $\theta=E{X_1} =0$ by $\bar{Y}_{n}=\bar{X}_n$. In the CLT at \eqref{eq:clt}, we have
\begin{align*}
\Sigma&=\text{Var}[X_1]+\sum_{k=1}^{\infty}[\text{Cov}(X_1,X_{1+k})+\text{Cov}(X_1,X_{1+k})^T]\\
&=(I_p-\Phi)^{-1}V+V(I_p-\Phi)^{-1}-V.
\end{align*}

A geometrically ergodic Markov chain is generated with $\Phi$ chosen as follows. Consider a $p\times p$ matrix $A$ with each entry generated from standard normal distribution,  and $B=AA^{T}$ which is a symmetric matrix with the largest eigenvalue $m$. Then $\Phi=B/(m+1)$ ensures geometrically ergodicity. The corresponding $\Sigma$ is estimated by weighted BM and SV estimators with window functions at \eqref{eq:bartlett}, \eqref{eq:th} and \eqref{eq:ft} using a truncation point of $\lfloor n^{1/3} \rfloor$.  Recursive estimators use default settings from the \texttt{rTACM} package while only storing the final variance estimate.  

For each combination of $p \in \{ 10, 20, 30 \}$ and $n \in \{ 1e5, 1e6, 5e6 \}$, Table~\ref{table:var1} presents average computing time over 10 replications.  There are significant computational gains for weighted BM with flat top and Bartlett lag windows.  However, there is minimal computation gain for the Tukey-Hanning lag window since the double sum in \eqref{eq:weighted} must be evaluated fully.  Table~\ref{table:var1} also shows recursive estimates are significantly slower.  When $p=30$ and $n=5e5$, recursive, flat top SV, and flat top weighted BM estimators take approximately 30 minutes, 60 seconds, and one second, respectively.  Increasing $n$ to $1e6$ iterations, the same estimators require approximately one hour, three minutes, and two seconds, respectively.  Computational gains continue to increase with further increases to $p$ or $n$ (see e.g.\ Section~\ref{sec:dynamic}).

\begin{table}[tbh]
\centering
\begin{tabular}{c|ccc|ccc|ccc}
\multicolumn{10}{c}{\textbf{Weighted BM}}\\
Window & \multicolumn{3}{c|}{Flat top} & \multicolumn{3}{c|}{BM} & \multicolumn{3}{c}{Tukey-Hanning} 
\\
$n$&$5e4$&$1e5$&$5e5$&$5e4$&$1e5$&$5e5$&$5e4$&$1e5$&$5e5$
\\ \hline
p=10& $<$.1 s  & .1 s   & .4 s & $<$.1 s  & $<$.1 s   & .2 s  & 1.9 s  &  4.8 s  &  27 s 
\\ 
p=20& .1 s & .2 s & .6 s& $<$.1 s & $<$.1 s & .4 s & 2.5 s  & 6 s & 36 s
 \\ 
 p=30& .1 s & .2 s & .9 s & $<$.1 s & .1 s & .6 s & 3 s & 6.8 s &  41 s
\end{tabular}
\begin{tabular}{c|ccc|ccc|ccc}
\multicolumn{10}{c}{\textbf{Spectral Variance}}\\
Window  & \multicolumn{3}{c|}{Flat top} & \multicolumn{3}{c|}{Bartlett} & \multicolumn{3}{c}{Tukey-Hanning} 
\\
$n$&$5e4$&$1e5$&$5e5$&$5e4$&$1e5$&$5e5$&$5e4$&$1e5$&$5e5$
\\ \hline
p=10& .8 s  & 2 s   & 13 s & .8 s  & 2.1 s  & 13 s  & .8 s & 1.9 s  & 12 s 
\\ 
p=20& 1.9 s & 5.5 s & 32 s& 1.9 s & 4.9 s & 33 s & 1.8 s & 4.5 s & 31 s
 \\ 
 p=30& 3.4 s & 7.9 s & 59 s & 3.5 s & 8.1 s & 59 s & 3.3 s & 7.7 s & 57 s
\end{tabular}
\begin{tabular}{c|ccc}
 \multicolumn{4}{c}{\textbf{Recursive}} \\
$n$&$5e4$&$1e5$&$5e5$
\\ \hline
p=10 & 42 s& 1.6 min  & 6.1 min
\\ 
p=20&1.7 min & 3.5 min & 14.9 min
 \\ 
 p=30&3.2 min & 6.3 min & 28.5 min 
\end{tabular}
\caption{Computational time in seconds (s) or minutes (min) for weighted BM, SV, and recursive estimators.  Monte Carlo standard errors are approximately 2\% of reported times.} \label{table:var1}
\end{table}

\section{Simulation studies} \label{sec:examples}
This section considers four examples to evaluate the finite sample properties of weighted BM estimators.  Our first two examples consider geometrically ergodic Markov chains generated from univariate and multivariate vector auto-regressive models.  The aim here is to compare weighted BM and SV estimators in terms of accuracy since the true value of $\Sigma$ is known.  Our final two examples compare performances of the two estimators on real datasets where the true values are unknown.

\subsection{AR(1) model}
Consider the autoregressive process of order 1 (AR(1)) where
\[
X_{t}=\phi X_{t-1}+\epsilon_{t} \quad \text{ for } t=1,2, \dots
\]
and $\epsilon_{t}$ are i.i.d.\ N(0,1).  For $|\phi|<1$ the Markov chain is geometrically ergodic.  Consider estimating $\theta=E[X_1]=0$ by $\bar{Y}_{n}=\bar{X}_n$, then in the CLT at \eqref{eq:clt} we have 
\[
\Sigma=\text{Var}[X_1]+2\sum_{k=1}^{\infty}\text{Cov}(X_1,X_{1+k})=1/(1-\phi)^2.
\]
A range of $\phi$ from $0.6$ to $0.9$ are evaluated and the true value of $\Sigma$ is used to compare weighted BM and SV estimators.

\begin{figure}[tbh]
    \centering
    \includegraphics[width=0.95\textwidth]{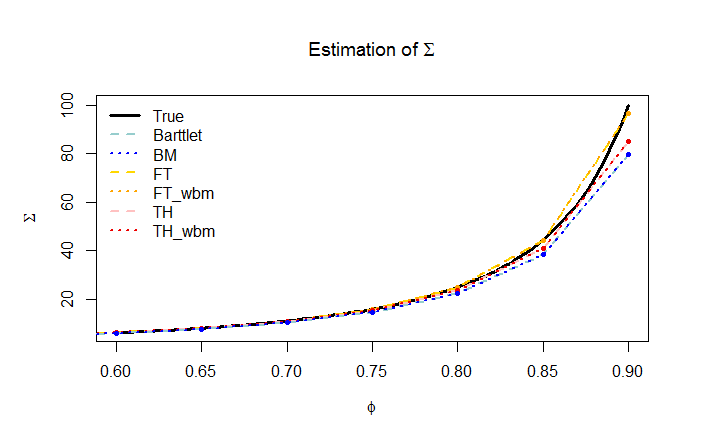}
    \caption{Estimation of $\Sigma$ for AR(1) model with $\phi$ between 0.6 and 0.9. Results are based on 500 independent replications with $n = 1e5$.}
    \label{fig:ar1est}
\end{figure}

For each $\phi$, generate an AR(1) Markov chain of length 1e5 and compute weighted BM and SV estimators using the three lag window functions at \eqref{eq:bartlett}, \eqref{eq:th} and \eqref{eq:ft}.  Truncation point $b$ equals to $\lfloor n^{1/3} \rfloor = 46$ for all six estimators. The procedure is repeated independently for 500 times and the average of 500 replications are shown in Figure~\ref{fig:ar1est}.  When the autocovariance is low, all the estimators perform well.  As $\phi$ increases, the flat top lag window outperforms Bartlett and Tukey-Hanning windows for both SV and weighted BM estimators.

\begin{figure}[tbh]
        \centering
        \begin{subfigure}[t]{0.5\textwidth}
                \includegraphics[width=\textwidth]{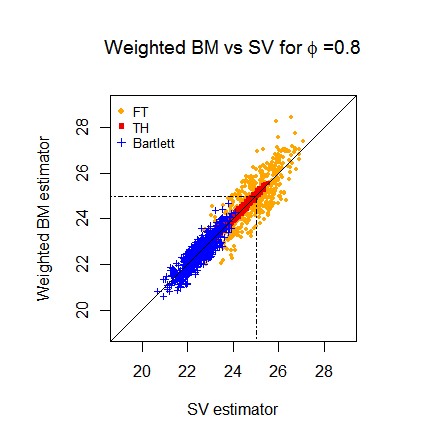}
        \end{subfigure}%
        \begin{subfigure}[t]{0.5\textwidth}
                \includegraphics[width=\textwidth]{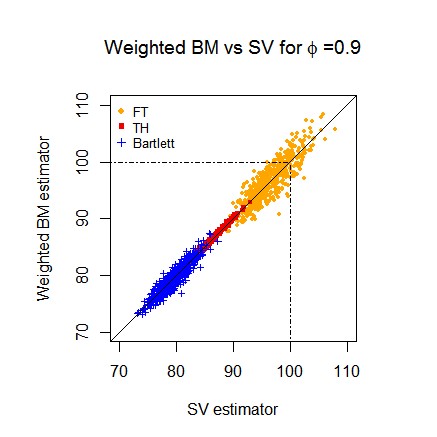}
        \end{subfigure}
        \caption{Weighted BM and SV estimators for $\phi = \{ 0.8, 0.9 \}$ where the true value of $\Sigma$ is denoted by a dashed line.}
        \label{fig:ar1est89}
\end{figure}

Weighted BM and SV estimates with the same lag window are very close as in Figure \ref{fig:ar1est}.  To explore this relationship in more detail, Figure~\ref{fig:ar1est89} plots SV against weighted BM estimates for the same Markov chain realizations when $\phi = \{ 0.8, 0.9 \}$.  Again we see the flat top lag window reduces bias but the variability increases slightly agreeing with our theoretical results.  Another interesting observation is the Tukey-Hanning window estimates are closer to the identity line implying its variance ratio is closer to one.  Figure~\ref{fig:ar1var} verifies this observation showing the Tukey-Hanning ratio is empirically very close to 1.  Figure~\ref{fig:ar1var} also shows the Bartlett and flat top windows variance ratios are close to the theoretical values of $1.5$ and $1.875$, respectively.

\begin{figure}[tbh]
    \centering
    \includegraphics[width=0.8\textwidth]{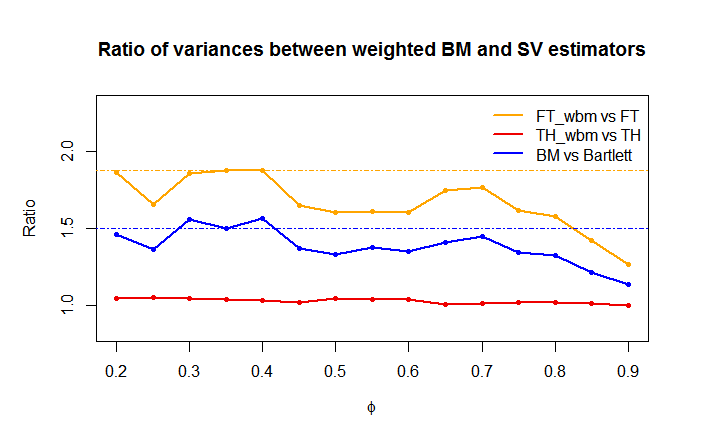}
    \caption{Variance ratios over 500 replications with $n = 1e5$.}
    \label{fig:ar1var}
\end{figure}

\subsection{Vector auto-regressive model}

Consider again the $p$-dimensional VAR(1) from Section~\ref{sec:comptime} where we are now interested in the accuracy of weighted BM and SV estimators.  We now obtain 50 independent replications for each combination of $p \in \{ 10, 20, 30 \}$ and $n \in \{ 1e5, 1e6, 5e6 \}$ for a given $\Phi_0=B/(m+0.1)$ with $B$ and $m$ as in Section~\ref{sec:comptime}.  Finally, let $\Phi=k\cdot\Phi_0$, where $k \in \{0.01, 0.2, 0.4, 0.6, 0.8 \}$. It is easy to see larger $k$ implies stronger autocovariance and cross-correlation in the chain. These $\Phi$ are used to generate geometrically ergodic Markov chains from which $\Sigma$ is estimated.  For an estimator $\hat{\Sigma}$, define $E=\hat{\Sigma}-\Sigma$ and consider mean squared error (MSE) over the entries of $E$ as a measurement of accuracy, i.e.\ 
\[
\text{MSE}=\dfrac{1}{p^2}\sum_i\sum_j e_{ij}^2.
\]
Figure~\ref{fig:var1mse} shows the averaged MSE ratio between weighted BM and SV estimators over 500 replications. Weighted BM estimators have inflated MSE but the ratios are below 2 across $k$.  Ratios for the Bartlett and Tukey-Hanning lag windows have a significant drop for $k=0.8$ resulting from large MSEs.  Flat top window estimators show less of this trend since they are more accurate when $k=0.8$.  Combining computation and accuracy information, weighted BM estimators with flat top lag windows exhibit superior performance.
\begin{figure}[tbh]
    \centering
    \includegraphics[width=0.8\textwidth]{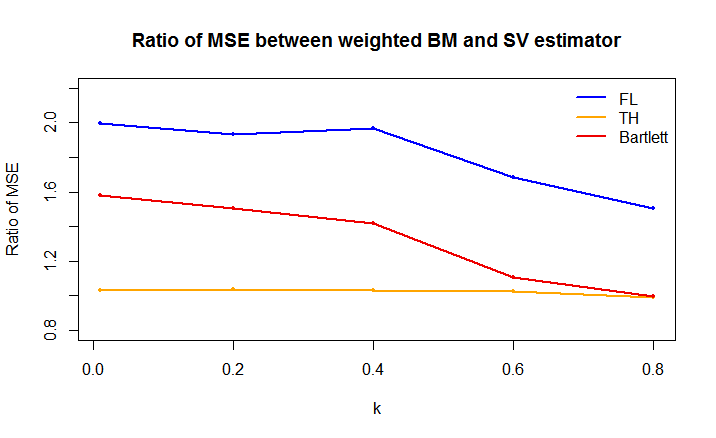}
    \caption{MSE ratio between weighted BM and SV methods.}
    \label{fig:var1mse}
\end{figure}

\subsection{Bayesian dynamic space-time model} \label{sec:dynamic}
Consider monthly temperature data collected at 10 nearby station in northeastern United States in 2000, a subset of {\tt NETemp} data described in {\tt R} package {\tt spBayes} \citep{finl:bana:carl:2007}. A Bayesian dynamic model proposed by \citep{gelf:bane:game:2005} is fitted to the data and the model treats time as discrete and space as continuous variable.

Suppose $y_t$ denotes the temperature observed at location $s$ and time $t$ for $s=1,2,..., N_s$ and $t=1,2,..., N_t$. Let $x_t(s)$ be a $k\times 1$ vector of predictors and $\beta_t$ be a $k\times 1$ coefficient vector, which is a purely time component, and $u_{t}(s)$ denotes a space-time component. The model is 
\[ y_t(s) = {\pmb x}_t (s)^{T} {\pmb \beta}_t + u_t (s) + \epsilon_t (s),\ \  \epsilon_t \sim N(0, \tau_t^2),  \]
\[ {\pmb \beta}_t = {\pmb \beta}_{t-1} + {\pmb \eta}_t; \ \ {\pmb \eta}_t \sim N_p (0, \Sigma_{\eta}), \]
\[ u_t(s) = u_{t-1}(s) + w_t(s); \ \ w_t(s) \sim GP (0, C_t(\cdot, \sigma_{t}^{2}, \phi_{t})) \]
where $GP (0, C_t(\cdot, \sigma_{t}^{2},\phi_t))$ is a spatial Gaussian process where $C_t(s_1, s_2; \sigma_t^2,\phi_t) = \sigma_t^2 \rho (s_1, s_2; \phi_t)$, $\rho (\cdot; \phi)$ is an exponential correlation function with $\phi$ controlling the correlation decay, and $\sigma^2_t$ represents the spatial variance components. The Gaussian spatial process allows closer locations to have higher correlations.  Time effects for both $\pmb \beta_t$ and $u_t(s)$ are characterized by transition equations, delivering a reasonable dependence structure.  Priors follow defaults in the {\tt spDynlM} function of the {\tt spBayes} package.  We are interested in estimating posterior expectations for 185 parameters denoted $\theta=(\pmb \beta_t,\ u_t(s),\ \sigma^2_t,\ \Sigma_\eta,\ \tau^2_t,\ \phi_t)$.  

\begin{table}[htb]
\begin{center}
\begin{tabular}{ccc|ccc|ccc}
\multicolumn{9}{c}{\textbf{Weighted BM}}\\
\multicolumn{3}{c|}{Flat top} & \multicolumn{3}{c|}{BM} & \multicolumn{3}{c}{Tukey-Hanning} 
\\
$5e4$&$1e5$&$2e5$&$5e4$&$1e5$&$2e5$&$5e4$&$1e5$&$2e5$
\\ \hline
.5 s  & .9 s   & 1.9 s & .3 s  & .6 s   & 1.3 s  & 11 s  &  25 s  &  56 s 
\\ 
\multicolumn{9}{c}{\textbf{Spectral Variance}}\\
\multicolumn{3}{c|}{Flat top} & \multicolumn{3}{c|}{Bartlett} & \multicolumn{3}{c}{Tukey-Hanning} 
\\
$5e4$&$1e5$&$2e5$&$5e4$&$1e5$&$2e5$&$5e4$&$1e5$&$2e5$
\\ \hline
52 s & 2.2 min & 5.6 min & 52 s & 2.2 min & 5.6 min & 51 s & 2.2 min & 5.6 min
\end{tabular}
\end{center}
\caption{Computational time in seconds (s) or minutes (min) for weighted BM and SV estimators in the Bayesian dynamic space-time model.}
\label{table:weather}
\end{table}


Again we consider Markov chains of length $5e4$, $1e5$ and $2e5$ and compute computational time ratios in Table~\ref{table:weather}.  For this high-dimensional Bayesian analysis, weighted BM estimators are much cheaper to compute for Bartlett and flat top lag windows.  Figure~\ref{fig:weather} plots estimates of the diagonal elements of $\Sigma$ obtained with weighted BM and SV methods on the log scale.  Since the points are close to the identity line it is clear both methods produce similar estimates.

\begin{figure}[htb]
        \centering
        \begin{subfigure}[b]{0.3\textwidth}
                \includegraphics[width=\textwidth]{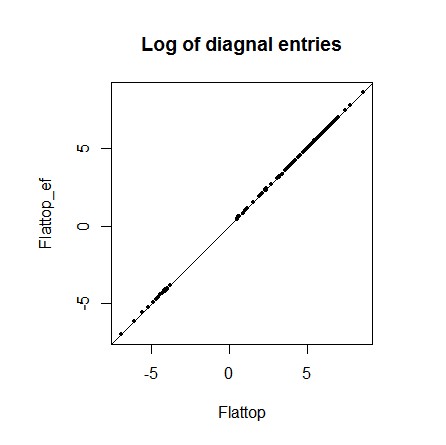}
                \caption{Flattop}
                \label{fig:gull}
        \end{subfigure}
        \begin{subfigure}[b]{0.3\textwidth}
                \includegraphics[width=\textwidth]{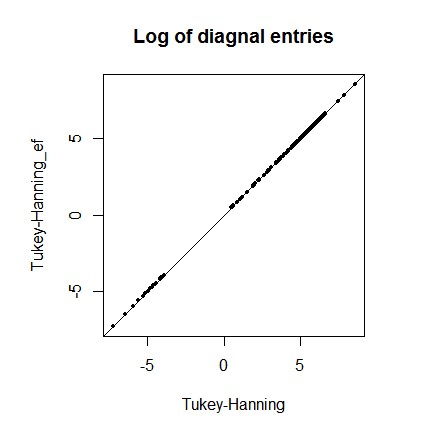}
                \caption{Tukey-Hanning}
                \label{fig:tiger}
        \end{subfigure}
        \begin{subfigure}[b]{0.3\textwidth}
                \includegraphics[width=\textwidth]{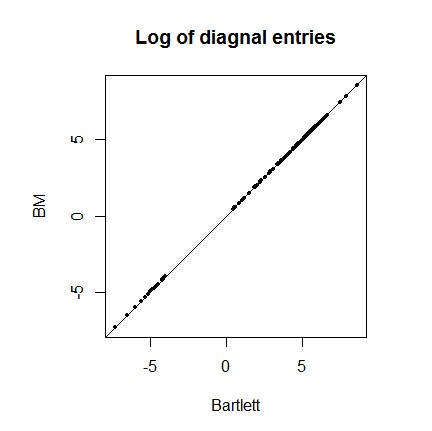}
                \caption{Bartlett}
                \label{fig:mouse}
        \end{subfigure}
        \caption{Estimates of the diagonal elements of $\Sigma$ obtained with weighted BM and SV methods on the log scale.}\label{fig:weather}
\end{figure}

\subsection{Bayesian logistic regression model}
Environmental data of 1000 site observations in New Zealand are considered to study the determinants of presence or absence of the short-finned eel (Anguilla australis) in {\tt R} {\tt dismo} package \citep[see e.g.][]{elit:leat:2008,hijm:2010}. Five continuous variables, SegSumT, DSDist, USNative, DSMaxSlope, and DSSlope, along with a categorical variable Method with five levels (Electric, Spo, Trap, Net, Mixture) are chosen as in \cite{leat:elit:2008} to predict Auguilla australis presence via a Bayesian logistic regression model.

For the $i$th observation, suppose $Y_i=1$ denotes presence and $Y_i=0$ denotes absence of Anguilla australis. Let $x_i$  be a $p\times 1$ covariate vector and ${\pmb \beta}=(\beta_0,\beta_1,...,\beta_9)$ be the $p\times 1$ coefficient vector where $p=10$. Consider the following Bayesian logistic regression model
\[
Y_i\sim Bernoulli(p_i) \text{ and } p_i\sim\dfrac{exp(x_i^T{\pmb \beta})}{1+exp(x_i^T{\pmb \beta})}.
\]
Priors for ${\pmb \beta}$ are chosen to be ${\pmb \beta}\sim N({\pmb 0},\ 100\mathbf{I}_k)$ as in \cite{boon:merr:krac:2014} and the {\tt MCMClogit} function in the {\tt MCMCpack} package is used to sample the Markov chain.

Using methods from \cite{vats:fleg:jone:2015output}, we construct a $90\%$ confidence region of the 10 dimensional parameter ${\pmb \beta}$ based on the BM estimator, the SV estimator with a flat top lag window, and the weighted BM estimator with a flat top lag window. Coverage probabilities of these confidence regions from 1000 repeated simulations are used to evaluate the performance of each method.  Since the true value of ${\pmb \beta}$ is unknown, the average of 500 chains each of length 1e6 is used as the ``truth''.  Table~\ref{table:dismo} shows the coverage probabilities for chains of length $1e4$, $5e4$, $1e5$, and $5e5$ and two different batch sizes.  When $b=\lfloor n^{1/3} \rfloor$, both estimators based on a flat top window are far superior to the BM estimator.  When $b= \lfloor n^{1/2} \rfloor$, the improvement remains but is less substantial.

\begin{table}[h]
\centering
\begin{tabular}{l|cccc|cccc}
 & \multicolumn{4}{c|}{$b=\lfloor n^{1/3} \rfloor$} & \multicolumn{4}{c}{$b=\lfloor n^{1/2} \rfloor$} 
\\
 &$1e4$&$5e4$&$1e5$&$5e5$&$1e4$&$5e4$&$1e5$&$5e5$
\\ \hline
Weighted BM &0.342   &0.657  & 0.741  &  0.853& 0.714  & 0.825  &0.849 & 0.868   
\\ 
BM &0.168 & 0.400 & 0.506& 0.705 & 0.643 & 0.806&0.835  & 0.868 
 \\ 
SV &0.368 & 0.655 & 0.739 &0.858 & 0.764 & 0.844&0.866 & 0.879
\end{tabular}
\caption{Observed coverage probabilities of confidence regions for the Bayesian logistic regression model.  Nominal level is 0.90 and Monte Carlo standard errors range from 0.01 to 0.016.
} \label{table:dismo}
\end{table}

\section{Discussion} \label{sec:discussion}
This paper considers a family of weighted BM estimators and obtains conditions for strong consistency. These estimators are fast to compute and comparable to existing SV estimators in terms of accuracy. Within this family, we advocate the flat top weighted BM estimator which is a computationally efficient robust estimator. Other estimators considered either require heavy computation or yield less desirable finite sample properties. By targeting both accuracy and computing effort, the advocated estimator is convenient to apply in modern high-dimensionally MCMC simulations.  The proposed estimator can be further incorporated in a sequential stopping rule, where strong consistency and computational efficiency are necessary.  Computational efficiency is also helpful when calculating valid asymptotic standard errors for the generalized importance sampling estimator \citep{roy:tan:fleg:2018, roy2018selection}.

Computational complexity of the flat top weighted BM estimator is $O(n)$ since it can be expressed as a difference between BM estimators \citep[see e.g.][]{alex:fish:seila:1997}.  Memory complexity will also be $O(n)$ if $b$ is allowed to to increase by ones since the entire chain must be stored.  \cite{gong:fleg:2016} propose a low-cost alternative sampling plan that satisfies conditions necessary for strong consistency.  Specifically, they set $b = \inf \left\{ 2^k : 2^k \ge n^{\nu}, k \in \mathbb{Z}^+ \right\}$ so that $b$ increases by doubling the batch size.  It is then possible to store only the batch means and merge every two batches when the batch size increases twofold.  Such a sampling plan reduces memory complexity to $O(a)$.  Moreover, the estimator can be updated in $O(1)$ computational steps using a recursive variance calculation.

The weighted BM estimator with a Tukey-Hanning window does not reduce computational time significantly. As mentioned earlier, the proposed estimators are more beneficial for lag window functions with $\Delta_{2}w_{n}(k)=0$ for certain $k$. Nevertheless, as dimension and chain length increases, weighted BM still save some computing time.  More efficient coding or a linear approximation could provide additional efficiency if one prefers the Tukey-Hanning window. 

All estimators in this article use the same truncation point for a fair comparison, which are not necessarily the best.  In fact, finite sample coverage probabilities should improve by choosing better truncation points as suggested by \cite{liu:fleg:2018}.  These optimal truncation points are those that minimize the asymptotic mean-squared error. If one wants to omit further exploration, we suggest using the same truncation point for weighted BM estimators as the corresponding SV estimators.  \cite{fleg:jone:2010} provide an illustrative example regarding optimal batch sizes for BM and overlapping BM. For flat top windows, \cite{poli:2003} suggests an empirical rule for the optimal truncation point selection.  However, optimal batch size is a direction of future research for weighted BM estimators.



\section*{Acknowledgments}

We are grateful to the anonymous referees, anonymous associate editor, Galin Jones, and Dootika Vats for their constructive comments and helpful discussions, which resulted in many improvements to this paper.

\begin{appendix}
\section{Preliminaries for Theorem~\ref{thm:consistency}}

We first introduce some notations and propositions. Recall $B=\{B(t),t\geq 0\}$ is a $p$-dimensional standard Brownian motion. Denote $\bar{B}=n^{-1}B(n)$, $\bar{B}_{l}(k)=k^{-1}[B(lk+k)-B(lk)]$ and $\dot{B}_{l}(k)=k^{-1}[B(l+k)-B(l)]$. The Brownian motion counterpart of $\hat{\Sigma}_{w}$ is 
\[
\tilde{\Sigma}_{w}=\sum_{k=1}^{b}\dfrac{1}{a_{k}-1}\sum_{l=0}^{a_{k}-1}k^{2}\Delta_{2}w_{n}(k)(\bar{B}_{l}(k)-\bar{B})(\bar{B}_{l}(k)-\bar{B})^{T}
\]
where the individual entries are denoted as $\tilde{\Sigma}_{w,ij}$.  

Recall $L$ is the lower triangular matrix satisfying $\Sigma=LL^{T}$ and denote $\Sigma_{ij}$ as the individual entries of $\Sigma$.  Finally define $C(t)=LB(t)$, $C^{(i)}(t)$ as the $i$th component of $C(t)$, $\bar{C}^{(i)}_{l}(k)=k^{-1}(C^{(i)}(l+k)-C^{(i)}(l))$, and $\bar{C}^{(i)}=n^{-1}C^{(i)}(n)$.

\begin{proposition} \label{prop:one} \citep[Corollary 1,][]{vats:fleg:jone:2018} Suppose Assumption~\ref{ass:batch} holds.  For all $\epsilon>0$ and for almost all sample paths there exists $n_{0}(\epsilon)$ such that for all $n\geq n_{0}$ and all $i=1,..., p$
\[
\left| C^{(i)}(n) \right| < (1+\epsilon)(2n\Sigma_{ii} \log \log n)^{1/2}.
\]
\end{proposition}

\begin{proposition} \label{prop:sup} \citep[Corollary 2,][]{vats:fleg:jone:2018} Suppose Assumption~\ref{ass:batch} holds.  For all $\epsilon>0$ and for almost all sample paths, there exists $n_{0}(\epsilon)$ such that for all $n\geq n_{0}$ and all $i=1,..., p$
\[
\left| \bar{C}^{(i)}_{l}(k) \right| \leq\dfrac{1}{k}\sup_{0\leq l\leq n-b}\sup_{0\leq s\leq b}|C^{(i)}(l+s)-C^{(i)}(l)|<\dfrac{1}{k}2(1+\epsilon)(b\Sigma_{ii} \log n)^{1/2}.
\]
\end{proposition}

\section{Proof of Theorem~\ref{thm:consistency}} \label{proof:theo1}

Theorem~\ref{thm:consistency} follows from Lemmas~\ref{lemma:tilde} and~\ref{lemma:hat} presented below.

\begin{lemma}\label{lemma:tilde}
Suppose Assumption~\ref{ass:batch} holds. If there exists a constant $c \geq 1$ such that $\sum_{n}(b/n)^{c}<\infty$ and \eqref{eq:cond1} holds, then $\tilde{\Sigma}_{w}\rightarrow\ I_{p}$ as $n \to \infty$ w.p.1 where $I_{p}$ is the $p\times p$ identity matrix.
\end{lemma}

\begin{proof}
First consider the diagonal elements of $\tilde{\Sigma}_{w}$ which go to 1 as $n \to \infty$.  For $i=j$
\begin{align*}
\tilde{\Sigma}_{w,ii}&=\sum_{k=1}^{b}\dfrac{a_{k}}{a_{k}-1}\left(\dfrac{1}{a_{k}}\sum_{l=0}^{a_{k}-1}k^{2}\Delta_{2}w_{n}(k)(\bar{B}_{l}^{(i)}(k)-\bar{B}^{(i)})^{2}\right)\\
&=\sum_{k=1}^{b}\dfrac{a_{k}}{a_{k}-1}\left(\dfrac{1}{a_{k}}\sum_{l=0}^{a_{k}-1}k^{2}\Delta_{2}w_{n}(k)(\bar{B}^{(i)}_{l}(k)^{2}+(\bar{B}^{(i)})^{2}-2\bar{B}^{(i)}_{l}(k)\bar{B}^{(i)})\right)
 \\
 &=\sum_{k=1}^{b}\dfrac{a_{k}}{a_{k}-1}\left(\dfrac{1}{a_{k}}\sum_{l=0}^{a_{k}-1}k^{2}\Delta_{2}w_{n}(k)\bar{B}^{(i)}_{l}(k)^{2}+\dfrac{1}{a_{k}}a_{k}(\bar{B}^{(i)})^{2}k^{2}\Delta_{2}w_{n}(k) \right. \\
& \quad \quad \quad \quad \quad \quad \left. - \dfrac{2}{a_{k}}\bar{B}^{(i)}k^{2}\Delta_{2}w_{n}(k)\sum_{l=0}^{a_{k}-1}\bar{B}^{(i)}_{l}(k)\right)\\
 &=\sum_{k=1}^{b}\dfrac{a_{k}}{a_{k}-1}\left(\dfrac{1}{a_{k}}\sum_{l=0}^{a_{k}-1}k^{2}\Delta_{2}w_{n}(k)\bar{B}^{(i)}_{l}(k)^{2}+(\bar{B}^{(i)})^{2}k^{2}\Delta_{2}w_{n}(k) \right. \\
& \quad \quad \quad \quad \quad \quad \left. -\dfrac{2}{a_{k}}\bar{B}^{(i)}k^{2}\Delta_{2}w_{n}(k)\dfrac{n}{k}\bar{B}^{(i)}\right)\\
  &=\sum_{k=1}^{b}\dfrac{a_{k}}{a_{k}-1}\left(\dfrac{1}{a_{k}}\sum_{l=0}^{a_{k}-1}k^{2}\Delta_{2}w_{n}(k)\bar{B}^{(i)}_{l}(k)^{2}-(\bar{B}^{(i)})^{2}k^{2}\Delta_{2}w_{n}(k)\right)\\
&=\sum_{k=1}^{b}\dfrac{a_{k}}{a_{k}-1}\left[k\Delta_{2}w_{n}(k)\left(\dfrac{1}{a_{k}}\sum_{l=0}^{a_{k}-1}k\bar{B}^{(i)}_{l}(k)^{2}-k(\bar{B}^{(i)})^{2}\right)\right].
\end{align*}
By the proof of Proposition 3.1 in \cite{dame:1994} (p.\ 507), as $n \to \infty$ w.p.1
\[
\dfrac{1}{a_{k}}\sum_{l=0}^{a_{k}-1}k\bar{B}^{(i)}_{l}(k)^{2}\rightarrow 1 \text{ and } k(\bar{B}^{(i)})^{2}\rightarrow 0.
\]
Therefore as $n \to \infty$ w.p.1
\[
\tilde{\Sigma}_{w,ii} \to \sum_{k=1}^{b}k\Delta_{2}w_{n}(k)
\]
and since \eqref{eq:cond1} holds, $\tilde{\Sigma}_{w,ii} \to 1$ as $n \to \infty$ w.p.1.

Now consider the off-diagonal elements of $\tilde{\Sigma}_{w}$ which go to 0 as $n \to \infty$.  Here $i\neq j$ and
\begin{align*}
\tilde{\Sigma}_{w,ij}&=\sum_{k=1}^{b}\dfrac{1}{a_{k}-1}\sum_{l=0}^{a_{k}-1}k^{2}\Delta_{2}w_{n}(k)(\bar{B}_{l}^{(i)}(k)-\bar{B}^{(i)})(\bar{B}_{l}^{(j)}(k)-\bar{B}^{(j)})\\
&=\sum_{k=1}^{b}k\Delta_{2}w_{n}(k)\dfrac{k}{a_{k}-1}\sum_{l=0}^{a_{k}-1} \left[ \bar{B}^{(i)}_{l}(k)\bar{B}^{(j)}_{l}(k)-\bar{B}^{(i)}_{l}(k)\bar{B}^{(j)} \right. \\
& \quad \quad \quad \quad \quad \quad \left. -\bar{B}^{(i)}\bar{B}^{(j)}_{l}(k)+\bar{B}^{(i)}\bar{B}^{(j)} \right].
\end{align*}
By Lemma 3 in \cite{vats:fleg:jone:2015output}, as $n \to \infty$ w.p.1
\[
\dfrac{k}{a_{k}-1}\sum_{l=0}^{a_{k}-1}[\bar{B}^{(i)}_{l}(k)\bar{B}^{(j)}_{l}(k)-\bar{B}^{(i)}_{l}(k)\bar{B}^{(j)}-\bar{B}^{(i)}\bar{B}^{(j)}_{l}(k)+\bar{B}^{(i)}\bar{B}^{(j)}]\rightarrow 0.
\]
Then since \eqref{eq:cond1} holds, $\tilde{\Sigma}_{w,ij} \to 0$ as $n \to \infty$ w.p.1.  Hence $\tilde{\Sigma}_{w}\to I_{p}$ as $n \to \infty$ w.p.1.
\end{proof}

\begin{lemma}\label{lemma:hat}
Let Assumptions~\ref{ass:sip} and~\ref{ass:batch} hold. If as $n\to \infty$,
\begin{equation} \label{eq:lemma2.1}
b\psi(n)^{2} \log n\left(\sum_{k=1}^{b}|\Delta_{2}w_{n}(k)|\right)^{2}\rightarrow 0,
\end{equation}
and
\begin{equation} \label{eq:lemma2.2}
\psi(n)^{2}\sum_{k=1}^{b}|\Delta_{2}w_{n}(k)|\rightarrow 0,
\end{equation}
then $\hat{\Sigma}_{w}\rightarrow L\tilde{\Sigma}_{w}L^{T}$ as $n \to \infty$ w.p.1.
\end{lemma}

\begin{proof}
We will show the result componentwise, i.e.\ that $\hat{\Sigma}_{w,ij} \to \Lambda_{ij}$ where $\Lambda_{ij}$ denotes $ij$ entry of the matrix $L\tilde{\Sigma}_{w}L^{T}$.  For ease of exposition, let $Y_{i}= g(X_{i}) - \theta$ and define $\bar{C}_{l}(k)=L\bar{B}_{l}(k)$ and $\bar{C}=L\bar{B}$. Then
\begin{align}
L\tilde{\Sigma}_{w}L^{T} &=\sum_{k=1}^{b}\dfrac{1}{a_{k}-1}\sum_{l=0}^{a_{k}-1}k^{2}\Delta_{2}w_{n}(k)L(\bar{B}_{l}(k)-\bar{B})(\bar{B}_{l}(k)-\bar{B})^{T}L^{T}\notag \\ 
&=\sum_{k=1}^{b}\dfrac{1}{a_{k}-1}\sum_{l=0}^{a_{k}-1}k^{2}\Delta_{2}w_{n}(k)(\bar{C}_{l}(k)-\bar{C})(\bar{C}_{l}(k)-\bar{C})^{T}. \label{eq:matrixLambda}
\end{align}
Define vectors
\begin{equation*}
A_{k} =k(\bar{Y}_{l}(k)-\bar{C}_{l}(k)), \; D_{k} =B(l+k)-B(l), \; E_{n,k} =k\bar{B}, \text{ and } F_{n,k} =k(\bar{Y}-\bar{C}).
\end{equation*}
Then it is easy to show 
\begin{align}
k(\bar{Y}_{l}^{(i)}(k)-\bar{Y}^{(i)})&=k(\bar{Y}_{l}^{(i)}(k)-\bar{Y}^{(i)}+\bar{C}_{l}^{(i)}(k)-\bar{C}_{l}^{(i)}(k)+\bar{C}^{(i)}-\bar{C}^{(i)}) \notag \\
&=k(\bar{Y}_{l}^{(i)}(k)-\bar{C}_{l}^{(i)}(k))+(k\bar{C}_{l}^{(i)}(k)-k\bar{C}^{(i)})-k(\bar{Y}^{(i)}-\bar{C}^{(i)}) \notag \\
&=A_{k}^{(i)}+(LD_{k})^{(i)}-(LE_{n,k})^{(i)}-F_{n,k}^{(i)}. \label{eq:barYdiff}
\end{align}
Using the definition of $\hat{\Sigma}_{w}$ at \eqref{eq:weighted} with \eqref{eq:matrixLambda} and \eqref{eq:barYdiff}, we have
\begin{align}
\hat{\Sigma}_{w,ij} & - \Lambda_{ij} \notag \\
= & \sum_{k=1}^{b}\dfrac{1}{a_{k}-1}\sum_{l-0}^{n-k}k^{2}\Delta_{2}w_{k} \left[ (\bar{Y}_{l}^{(i)}(k)-\bar{Y}^{(i)})(\bar{Y}_{l}^{(j)}(k)-\bar{Y}^{(j)}) \right. \notag \\
& \quad \quad \quad \quad \quad \quad \left. -(\bar{C}_{l}^{(i)}(k)-\bar{C}^{(i)})(\bar{C}_{l}^{(j)}(k)-\bar{C}^{(j)}) \right] \notag \\
= & \sum_{k=1}^{b}\dfrac{1}{a_{k}-1}\sum_{l-0}^{n-k}\Delta_{2}w_{k} \left[ k(\bar{Y}_{l}^{(i)}(k)-\bar{Y}^{(i)}) k(\bar{Y}_{l}^{(j)}(k)-\bar{Y}^{(j)}) \right. \notag \\
& \quad \quad \quad \quad \quad \quad \left. -k(\bar{C}_{l}^{(i)}(k)-\bar{C}^{(i)}) k(\bar{C}_{l}^{(j)}(k)-\bar{C}^{(j)}) \right] \notag \\
= & \sum_{k=1}^{b}\dfrac{1}{a_{k}-1}\sum_{l-0}^{n-k} \Delta_{2}w_{k} \left[ \left( A_{k}^{(i)}+(LD_{k})^{(i)}-(LE_{n,k})^{(i)}-F_{n,k}^{(i)} \right) \right. \notag \\
& \quad \quad \quad \quad \quad \quad \left. \left( A_{k}^{(j)}+(LD_{k})^{(j)}-(LE_{n,k})^{(j)}-F_{n,k}^{(j)} \right) \right. \notag \\
& \quad \quad \quad \quad \quad \quad \left. - \left( (LD_{k})^{(i)}-(LE_{n,k})^{(i)} \right) \left( (LD_{k})^{(j)}-(LE_{n,k})^{(j)} \right) \right] \notag \\
= &\sum_{k=1}^{b}\dfrac{1}{a_{k}-1}\sum_{l=0}^{n-k} \Delta_{2}w_{k}  \left[ A_{k}^{(i)}A_{k}^{(j)}+A_{k}^{(i)}(LD_{k})^{(j)}-A_{k}^{(i)}(LE_{n,k})^{(j)}-A_{k}^{(i)}F_{n,k}^{(j)} \right. \notag\\
& + (LD_{k})^{(i)}A_{k}^{(j)}-(LD_{k})^{(i)}F_{n,k}^{(j)}-(LE_{n,k})^{(i)}A_{k}^{(j)}+(LE_{n,k})^{(i)}F_{n,k}^{(j)}\notag\\
& \left. - F_{n,k}^{(i)}A_{k}^{(j)}-F_{n,k}^{(i)}(LD_{k})^{(j)}+F_{n,k}^{(i)}(LE_{n,k})^{(j)}+F_{n,k}^{(i)}F_{n,k}^{(j)} \right] \notag\\
= & \sum_{k=1}^{b}\dfrac{1}{a_{k}-1}\sum_{l=0}^{n-k} \Delta_{2}w_{k} \left[ A_{k}^{(i)}A_{k}^{(j)}+F_{n,k}^{(i)}F_{n,k}^{(j)}+ \left( A_{k}^{(i)}(LD_{k})^{(j)}+(LD_{k})^{(i)}A_{k}^{(j)} \right) \right. \notag\\
&- \left( A_{k}^{(i)}(LE_{n,k})^{(j)}+(LE_{n,k})^{(i)}A_{k}^{(j)} \right) - \left( A_{k}^{(i)}F_{n,k}^{(j)}+F_{n,k}^{(i)}A_{k}^{(j)} \right) \label{eq:messy}\\
&\left. - \left( (LD_{k})^{(i)}F_{n,k}^{(j)}+F_{n,k}^{(i)}(LD_{k})^{(j)} \right) + \left( (LE_{n,k})^{(i)}F_{n,k}^{(j)}+F_{n,k}^{(i)}(LE_{n,k})^{(j)} \right) \right] \notag.
\end{align}
Taking absolute value of \eqref{eq:messy} 
\begin{align}
& \left| \hat{\Sigma}_{w,ij} - \Lambda_{ij} \right| \le \notag\\ 
& \sum_{k=1}^{b}\dfrac{1}{a_{k}-1}\sum_{l=0}^{n-k} \left| \Delta_{2}w_{k} \right| \left[ \left| A_{k}^{(i)}A_{k}^{(j)} \right| + \left| F_{n,k}^{(i)}F_{n,k}^{(j)} \right| + \left| A_{k}^{(i)}(LD_{k})^{(j)} + (LD_{k})^{(i)}A_{k}^{(j)} \right| \right. \notag\\
& \quad \quad \quad + \left| A_{k}^{(i)}(LE_{n,k})^{(j)} + (LE_{n,k})^{(i)}A_{k}^{(j)} \right| + \left| A_{k}^{(i)}F_{n,k}^{(j)} + F_{n,k}^{(i)}A_{k}^{(j)} \right| \label{eq:main.inequality}\\
& \quad \quad \quad \left. + \left| (LD_{k})^{(i)}F_{n,k}^{(j)} + F_{n,k}^{(i)}(LD_{k})^{(j)} \right| + \left| (LE_{n,k})^{(i)}F_{n,k}^{(j)} + F_{n,k}^{(i)}(LE_{n,k})^{(j)} \right| \right] \notag.
\end{align}
We will show each of the seven terms in \eqref{eq:main.inequality} goes to 0 as $n \to \infty$ w.p.1.  First we establish the following useful inequality.  From \eqref{eq:sip} in Assumption~\ref{ass:sip}, for any component $i$ and sufficiently large $n$,
\begin{equation} \label{eq:d.psi}
\left| \sum_{t=1}^{n} Y_{t}^{(i)} - C^{(i)}(n) \right| \le D \psi(n) .
\end{equation}

\begin{enumerate}
\item For any component $i$, we have
\begin{align}
\left| A_{k}^{(i)} \right| & =k[\bar{Y}^{(i)}_{l}(k)-\bar{C}_{l}^{(i)}(k)] \notag \\
&=k\left[k^{-1}\sum_{t=1}^{k}Y^{(i)}_{lk+t}-k^{-1}(C^{(i)}(lk+k)-C^{(i)}(lk))\right] \notag \\
&=\left[\sum_{t=1}^{(lk+k)}Y_{t}^{(i)}-\sum_{t=1}^{lk}Y^{(i)}_{t}\right]-[C^{(i)}(lk+k)-C^{(i)}(lk)] \notag \\
&=\left[\sum_{t=1}^{lk+k}Y_{t}^{(i)}-C^{(i)}(lk+k)\right]-\left[\sum_{t=1}^{lk}Y_{t}^{(i)}-C^{(i)}(lk)\right] \notag \\
&\leq 2D\psi(n), \label{eq:A.bound}
\end{align}
where the inequality is from \eqref{eq:d.psi} since $lk < lk+k \le n$.  Then using \eqref{eq:lemma2.2} and \eqref{eq:A.bound} as $n \to \infty$ w.p.1
\begin{align*}
\sum_{k=1}^{b}\dfrac{1}{a_{k}-1}\sum_{l=0}^{a_{k}-1} \left| \Delta_{2}w_{n}(k) \right| \left| A_{k}^{(i)}A_{k}^{(j)} \right| \leq 4D^{2}\psi^{2}(n)\sum_{k=1}^{b}\dfrac{a_{k}}{a_{k}-1}|\Delta_{2}w_{n}(k)|\rightarrow 0.
\end{align*}

\item For any component $i$,
\begin{align}
\left| F_{n,k}^{(i)} \right| & = \left| k(\bar{Y}-\bar{C}) \right| \notag \\
 & = \frac{k}{n} \left| \sum_{t=1}^{n} Y_{t}^{(i)} - C^{(i)}(n) \right| \notag \\
 & \le \frac{k}{n} D \psi(n) , \label{eq:F.bound}
\end{align}
where the inequality is from \eqref{eq:d.psi}.  Using \eqref{eq:lemma2.2}, \eqref{eq:F.bound}, and Assumption~\ref{ass:batch}, as $n \to \infty$ w.p.1
\begin{align*}
\sum_{k=1}^{b}\dfrac{1}{a_{k}-1}\sum_{l=0}^{a_{k}-1} \left| \Delta_{2}w_{n}(k) \right| \left| F_{n,k}^{(i)}F_{n,k}^{(j)} \right| \leq \frac{b^2}{n^2} D^{2}\psi^{2}(n)\sum_{k=1}^{b}\dfrac{a_{k}}{a_{k}-1}|\Delta_{2}w_{n}(k)|\rightarrow 0.
\end{align*}

\item For any component $i$, using Proposition~\ref{prop:sup}
\begin{align}
\left| (LD_{k})^{(i)} \right| & = \left| (LB(l+k))^{(i)} -  (LB(l))^{(i)} \right| \notag \\
 & = \left| C^{(i)}(l+k) -  C^{(i)}(l) \right| \notag \\
 & \le \sup_{0\leq l\leq n-b}\sup_{0\leq s\leq b} \left| C^{(i)}(l+s)-C^{(i)}(l) \right| \notag \\
 & \le 2(1+\epsilon)(b\Sigma_{ii} \log n)^{1/2} \label{eq:LD.bound} .
\end{align}
Then \eqref{eq:A.bound} and \eqref{eq:LD.bound} imply
\begin{align*}
\sum_{k=1}^{b} & \dfrac{1}{a_{k}-1}\sum_{l=0}^{a_{k}-1} \left| \Delta_{2}w_{n}(k) \right| \left| A_{k}^{(i)}(LD_{k})^{(j)} + (LD_{k})^{(i)}A_{k}^{(j)} \right| \\
 & \le  2 \left[ 2(1+\epsilon)(b\Sigma_{ii} \log n)^{1/2}\right] \left[ 2 D\psi(n) \right] \sum_{k=1}^{b}\dfrac{a_{k}}{a_{k}-1}|\Delta_{2}w_{n}(k)|,
\end{align*}
which tends to 0 as $n \to \infty$ w.p.1 by \eqref{eq:lemma2.1}.

\item For any component $i$, using Proposition~\ref{prop:one}
\begin{align}
\left| (LE_{n,k})^{(i)} \right| & = \frac{k}{n} \left| C^{(i)}(n) \right| \notag \\
 & \le \frac{k}{n} (1+\epsilon)(2n\Sigma_{ii} \log \log n)^{1/2} \label{eq:LE.bound} .
\end{align}
The using \eqref{eq:A.bound} and \eqref{eq:LE.bound}
\begin{align*}
\sum_{k=1}^{b} & \dfrac{1}{a_{k}-1}\sum_{l=0}^{a_{k}-1} \left| \Delta_{2}w_{n}(k) \right| \left| A_{k}^{(i)}(LE_{n,k})^{(j)} + (LE_{n,k})^{(i)}A_{k}^{(j)} \right| \\
 & \le 2 \left[ \frac{b}{n} (1+\epsilon)(2n\Sigma_{ii} \log \log n)^{1/2} \right] \left[ 2 D\psi(n) \right] \sum_{k=1}^{b}\dfrac{a_{k}}{a_{k}-1}|\Delta_{2}w_{n}(k)| \\
 & \le 8 D \Sigma_{ii}^{1/2}  (1+\epsilon) n^{-1/2} b \psi(n) \log n \sum_{k=1}^{b}\dfrac{a_{k}}{a_{k}-1}|\Delta_{2}w_{n}(k)|,
\end{align*}
which tends to 0 as $n \to \infty$ w.p.1 by \eqref{eq:lemma2.1}.

\item By \eqref{eq:A.bound} and \eqref{eq:F.bound}
\begin{align*}
\sum_{k=1}^{b} & \dfrac{1}{a_{k}-1}\sum_{l=0}^{a_{k}-1} \left| \Delta_{2}w_{n}(k) \right| \left| A_{k}^{(i)}F_{n,k}^{(j)} + F_{n,k}^{(i)}A_{k}^{(j)} \right| \\
 & \le 2 \left[ 2D\psi(n) \right] \left[ \frac{b}{n} D \psi(n) \right] \sum_{k=1}^{b}\dfrac{a_{k}}{a_{k}-1}|\Delta_{2}w_{n}(k)| \\
 & = 4 D^2 \frac{b}{n} \psi(n)^2 \sum_{k=1}^{b}\dfrac{a_{k}}{a_{k}-1}|\Delta_{2}w_{n}(k)|,
\end{align*}
which tends to 0 as $n \to \infty$ w.p.1 by \eqref{eq:lemma2.1}.

\item By \eqref{eq:F.bound} and \eqref{eq:LD.bound}
\begin{align*}
\sum_{k=1}^{b} & \dfrac{1}{a_{k}-1}\sum_{l=0}^{a_{k}-1} \left| \Delta_{2}w_{n}(k) \right| \left| (LD_{k})^{(i)}F_{n,k}^{(j)} + F_{n,k}^{(i)}(LD_{k})^{(j)} \right| \\
 & \le 2 \left[ \frac{b}{n} D \psi(n) \right] \left[ 2(1+\epsilon)(b\Sigma_{ii} \log n)^{1/2}\right] \sum_{k=1}^{b}\dfrac{a_{k}}{a_{k}-1}|\Delta_{2}w_{n}(k)| \\
 & = 4 D \Sigma_{ii}^{1/2}  (1+\epsilon) \frac{b}{n} (b \log n)^{1/2} \psi(n) \sum_{k=1}^{b}\dfrac{a_{k}}{a_{k}-1}|\Delta_{2}w_{n}(k)|,
\end{align*}
which tends to 0 as $n \to \infty$ w.p.1 by \eqref{eq:lemma2.1}.

\item By \eqref{eq:F.bound} and \eqref{eq:LE.bound}
\begin{align*}
\sum_{k=1}^{b} & \dfrac{1}{a_{k}-1}\sum_{l=0}^{a_{k}-1} \left| \Delta_{2}w_{n}(k) \right| \left| (LE_{n,k})^{(i)}F_{n,k}^{(j)} + F_{n,k}^{(i)}(LE_{n,k})^{(j)} \right| \\
 & \le 2 \left[ \frac{b}{n} D \psi(n) \right] \left[ \frac{b}{n} (1+\epsilon)(2n\Sigma_{ii} \log \log n)^{1/2} \right] \sum_{k=1}^{b}\dfrac{a_{k}}{a_{k}-1}|\Delta_{2}w_{n}(k)| \\
& \le 4 D \Sigma_{ii}^{1/2}  (1+\epsilon) \frac{b^{3/2}}{n^{3/2}} (b \log n)^{1/2} \psi(n) \sum_{k=1}^{b}\dfrac{a_{k}}{a_{k}-1}|\Delta_{2}w_{n}(k)|,
\end{align*}
which tends to 0 as $n \to \infty$ w.p.1 by \eqref{eq:lemma2.1}.
\end{enumerate}
Since each of the seven terms in \eqref{eq:main.inequality} goes to 0 as $n \to \infty$ w.p.1, $\hat{\Sigma}_{w}\rightarrow L\tilde{\Sigma}_{w}L^{T}$ as $n \to \infty$ w.p.1.
\end{proof}

\section{Preliminaries for Theorem~\ref{thm:var}}

The proof of Theorem~\ref{thm:var} uses the results in Corollary~\ref{cor:lemmaf}, Lemma~\ref{lemma:w} and Lemma~\ref{lemma:Esq0}. 

Consider the general family of Bartlett flat top SV estimators introduced by \cite{poli:roma:1995,poli:roma:1996} defined as 
\[
\hat{\sigma}^{2}_{ft}=\dfrac{1}{1-c}\dfrac{b}{n}\sum_{l=0}^{n-b}(\dot{Y}_{l}(b)-\bar{Y})^{2}-\dfrac{c}{1-c}\dfrac{cb}{n}\sum_{l=0}^{n-cb}(\dot{Y}_{l}(cb)-\bar{Y})^{2},
\]
where $0\leq c\leq 1$. When $c=1/2$, the resulting estimator is the SV estimator based on lag window at \eqref{eq:ft} denoted previously $\hat{\sigma}^2_{f}$.  Further define the Brownian motion expression of $\hat{\sigma}^2_{ft}$ as
\[
\tilde{\sigma}^{2}_{ft}=\dfrac{1}{1-c}\dfrac{b}{n}\sum_{l=0}^{n-b}(\dot{B}_{l}(b)-\bar{B})^{2}-\dfrac{c}{1-c}\dfrac{cb}{n}\sum_{l=0}^{n-cb}(\dot{B}_{l}(cb)-\bar{B})^{2}
\]
whose variance is given by Lemma~\ref{lemma:f}. Denote $\lim_{n\rightarrow \infty} f(n)/g(n)=0$ by $f(n)=o(g(n))$.

\begin{lemma}\label{lemma:f}
Under Assumption \ref{ass:batch},
\[\dfrac{n}{b}Var[\tilde{\sigma}^{2}_{ft}]=\left(\dfrac{8}{3}c+\dfrac{4}{3}\right)+o(1).\]
\end{lemma}

\begin{proof}
Note $Var[\tilde{\sigma}^{2}_{ft}]=E[\tilde{\sigma}^{4}_{ft}]-(E[\tilde{\sigma}^{2}_{ft}])^{2}$ and first consider
\begin{align}\label{eq:A1+A2+A3_f}
E&[\tilde{\sigma}^{4}_{ft}] =E\left[\Big(\dfrac{1}{1-c}\dfrac{b}{n}\sum_{l=0}^{n-b}(\dot{B}_{l}(b)-\bar{B})^{2}-\dfrac{c}{1-c}\dfrac{cb}{n}\sum_{l=0}^{n-cb}(\dot{B}_{l}(cb)-\bar{B})^{2}\Big)^{2}\right]\nonumber\\
&=E\left[\Big(\dfrac{1}{1-c}\Big)^{2}\dfrac{b^{2}}{n^{2}}\Big(\sum_{l=0}^{n-b}(\dot{B}_{l}(b)-\bar{B})^{2}\Big)^{2}+\Big(\dfrac{c}{1-c}\Big)^{2}\dfrac{(cb)^{2}}{n^{2}}\Big(\sum_{l=0}^{n-cb}(\dot{B}_{l}(cb)-\bar{B})^{2}\Big)^{2}\right.\nonumber\\
&\left.-\dfrac{2c^{2}}{(1-c)^{2}}\dfrac{b^{2}}{n^{2}}\Big(\sum_{l=0}^{n-b}(\dot{B}_{l}(b)-\bar{B})^{2}\Big)\Big(\sum_{l=0}^{n-cb}(\dot{B}_{l}(cb)-\bar{B})^{2}\Big)\right]\nonumber\\
&=A_{1}+A_{2}+A_{3}
\end{align}
where 
\begin{equation}\label{eq:A1_f}
A_{1}=E\left[\Big(\dfrac{1}{1-c}\Big)^{2}\dfrac{b^{2}}{n^{2}}\Big(\sum_{l=0}^{n-b}(\dot{B}_{l}(b)-\bar{B})^{2}\Big)^{2}\right],
\end{equation}

\begin{equation}\label{eq:A2_f}
A_{2}=E\left[\Big(\dfrac{c}{1-c}\Big)^{2}\dfrac{(cb)^{2}}{n^{2}}\Big(\sum_{l=0}^{n-cb}(\dot{B}_{l}(cb)-\bar{B})^{2}\Big)^{2}\right], \text{ and}
\end{equation}

\begin{equation}\label{eq:A3_f}
A_{3}=E\left[-\dfrac{2c^{2}}{(1-c)^{2}}\dfrac{b^{2}}{n^{2}}\Big(\sum_{l=0}^{n-b}(\dot{B}_{l}(b)-\bar{B})^{2}\Big)\Big(\sum_{l=0}^{n-cb}(\dot{B}_{l}(cb)-\bar{B})^{2}\Big)\right].
\end{equation}
Denote
\begin{equation}\label{eq:a1_f}
a_1=\sum_{l=0}^{n-b}(\dot{B}_{l}(b)-\bar{B})^{4},
\end{equation}

\begin{equation}\label{eq:a2_f}
a_2=\sum_{s=1}^{b-1}\sum_{l=0}^{n-b-s}(\dot{B}_{l}(b)-\bar{B})^{2}(\dot{B}_{l+s}(b)-\bar{B})^{2}, \text{ and}
\end{equation}

\begin{equation}\label{eq:a3_f}
a_3=\sum_{s=b}^{n-b}\sum_{l=0}^{n-b-s}(\dot{B}_{l}(b)-\bar{B})^{2}(\dot{B}_{l+s}(b)-\bar{B})^{2}.
\end{equation}
Then $A_{1}$ can be expressed as
\begin{align}\label{eq:A1_long_f}
A_{1}&=\dfrac{1}{(1-c)^{2}}\dfrac{b^{2}}{n^{2}}E\Bigg[\sum_{l=0}^{n-b}(\dot{B}_{l}(b)-\bar{B})^{4}\nonumber\\
&\ \ \ \ \ \ \ \ \ \ \ \ \ \ \ \ \  +2\sum_{s=1}^{b-1}\sum_{l=0}^{n-b-s}(\dot{B}_{l}(b)-\bar{B})^{2}(\dot{B}_{l+s}(b)-\bar{B})^{2}\nonumber\\
&\ \ \ \ \ \ \ \ \ \ \ \ \ \ \ \ \ \ +2\sum_{s=b}^{n-b}\sum_{l=0}^{n-b-s}(\dot{B}_{l}(b)-\bar{B})^{2}(\dot{B}_{l+s}(b)-\bar{B})^{2}\Bigg]\nonumber\\
&=\dfrac{1}{(1-c)^{2}}\dfrac{b^{2}}{n^{2}}E[a_1+2a_2+2a_3].
\end{align}
To calculate $E[a_1]$ at \eqref{eq:a1_f}, consider $E[(\dot{B}_{l}(b)-\bar{B})^{4}]$. Let $U_{t}=B(t)-B(t-1)$ where $t=1,2,..., n$ and note $U_{t}$ are i.i.d.\ $ N(0,1)$. Then for $l=0,...,(n-b)$,
\[\dot{B}_{l}(b)-\bar{B}=\dfrac{(n-b)}{nb}\sum_{t=l+1}^{l+b}U_{t}-\dfrac{1}{n}\sum_{t=1}^{l}U_{t}-\dfrac{1}{n}\sum_{t=l+b+1}^{n}U_{t}.\]
Notice 
\[\dot{B}_{l}(b)-\bar{B}\sim N(0,(n-b)/bn)\]
since
\[E\left[\dot{B}_{l}(b)-\bar{B}\right]=0\] 
and 
\[Var[\dot{B}_{l}(b)-\bar{B}]=\Big(\dfrac{n-b}{nb}\Big)^{2}b+\dfrac{n-b}{n^{2}}=\dfrac{n-b}{bn}.\]
Then 
\[bn/(n-b)(\dot{B}_{l}(b)-\bar{B})^{2}\sim \chi^{2}_{(1)}\]
with
\begin{equation}\label{eq:E(B)^2=1}
E\left[\dfrac{bn}{n-b}(\dot{B}_{l}(b)-\bar{B})^{2}\right]=1
\end{equation}
and
\[\ Var\left[\dfrac{bn}{n-b}(\dot{B}_{l}(b)-\bar{B})^{2}\right]=2.\]
Therefore
\begin{equation}\label{eq:E(B)^4}
E[(\dot{B}_{l}(b)-\bar{B})^{4}]=(E[(\dot{B}_{l}(b)-\bar{B})^{2}])^{2}+Var[(\dot{B}_{l}(b)-\bar{B})^{2}]=3\Big(\dfrac{n-b}{bn}\Big)^{2}.
\end{equation}
By \eqref{eq:a1_f},
\begin{equation}\label{eq:Ea1_f}
E[a_1]=\sum_{l=0}^{n-b}E[(\dot{B}_{l}(b)-\bar{B})^{4}]=3(n-b+1) \left(\dfrac{n-b}{bn}\right)^{2}.
\end{equation}
Define $Z_{1}=(\dot{B}_{l}(b)-\bar{B})$ and $Z_{2}=(\dot{B}_{l+s}(b)-\bar{B})$ for $l=0,...,(n-b-s)$ and $s=1,...,(b-1).$ Then \eqref{eq:a2_f} can be approached by
\[E[a_2]=\sum_{s=1}^{b-1}\sum_{l=0}^{n-b-s}E[Z_1^2Z_2^2].\]
We will obtain $E[Z_1^2Z_2^2]$ through the joint distribution of $Z=(Z_{1},Z_{2})^{T}$. Since $Z_{1}$ and $Z_{2}$ are linear combinations of i.i.d. standard normal variables, denote $U=(U_{1} , \dots, U_{n})^{T}$,  then $Z=VU$ where
\[V = \begin{bmatrix} -\dfrac{1}{n}&\cdots&-\dfrac{1}{n} &\dfrac{n-b}{bn}&\cdots&\dfrac{n-b}{bn}&-\dfrac{1}{n}&\cdots&\cdots&-\dfrac{1}{n}  \\
-\dfrac{1}{n}&\cdots&\cdots&-\dfrac{1}{n}&\dfrac{n-b}{bn}&\cdots&\dfrac{n-b}{bn}&-\dfrac{1}{n}&\cdots&-\dfrac{1}{n}\end{bmatrix}. \]
The joint distribution of $Z$ is 
\[\left[ \begin{matrix} Z_1 \\ Z_2 \end{matrix} \right] \sim N\left( \begin{bmatrix} 0\\ 0\end{bmatrix},\ \left[\begin{matrix}\dfrac{n-b}{bn}& \dfrac{nb-ns-b^{2}}{nb^{2}}\\ \dfrac{nb-ns-b^{2}}{nb^{2}}& \dfrac{n-b}{bn}\end{matrix}\right]\right).\]
Recall if 
\[\left[ \begin{matrix} Y_1 \\ Y_2 \end{matrix} \right] \sim N \left( \begin{bmatrix} \mu_{1}\\ \mu_{2}\end{bmatrix},\ \left[ \begin{matrix} \Sigma_{11}&\Sigma_{12} \\ \Sigma_{21}&\Sigma_{22} \end{matrix} \right] \right),\]
and $\Sigma_{22}$ is non-singular, then the conditional distribution of $Y_{1}|Y_{2}$ is
\[Y_{1}|Y_{2}\sim N(\mu_{1}+\Sigma_{12}\Sigma_{22}^{-1}(Y_{2}-\mu_{2}),\ \Sigma_{11}-\Sigma_{12}\Sigma_{22}^{-1}\Sigma_{21}).\]
In our case,
\[Z_{1}|Z_{2}\sim N\left(\dfrac{b(n-b)-ns}{b(n-b)}Z_{2},\ \dfrac{2bs(n-b)-ns^{2}}{b^{3}(n-b)}\right) \text{ and} \]
\[Z_{2}\sim N\left(0,\ \dfrac{n-b}{bn}\right).\]
Using iterated expectations
\begin{align*}
E[Z_{1}^{2}Z_{2}^{2}]&=E_{Z_{2}}[E_{Z_{1}|Z_{2}}[Z_{1}^{2}Z_{2}^{2}|Z_{2}]]\\
&=E_{Z_{2}}[Z_{2}^{2}E_{Z_{1}|Z_{2}}[Z_{1}^{2}|Z_{2}]]\\
&=E_{Z_{2}}\left[Z_{2}^{2}\left(\left(\dfrac{b(n-b)-ns}{b(n-b)}Z_{2}\right)^{2}+\dfrac{2bs(n-b)-ns^{2}}{b^{3}(n-b)}\right)\right]\\
&=\left(\dfrac{b(n-b)-ns}{b(n-b)}\right)^{2}E_{Z_{2}}[Z_{2}^{4}]+\dfrac{2bs(n-b)-ns^{2}}{b^{3}(n-b)}E_{Z_{2}}[Z_{2}^{2}]\\
&=\left(\dfrac{b(n-b)-ns}{b(n-b)}\right)^{2}3\left(\dfrac{n-b}{bn}\right)^{2}+\dfrac{2bs(n-b)-ns^{2}}{b^{3}(n-b)}\left(\dfrac{n-b}{bn}\right)\\
&=\dfrac{3(b(n-b)-ns)^{2}+2nbs(n-b)-n^{2}s^{2}}{b^{4}n^{2}}\\
&=\dfrac{1}{b^{4}n^{2}}(3b^{2}(n-b)^{2}+3n^{2}s^{2}-6nbs(n-b)+2nbs(n-b)-n^{2}s^{2})\\
&=\dfrac{2}{b^{4}}s^{2}+\left(\dfrac{4}{nb^{2}}-\dfrac{4}{b^{3}}\right)s+\left(\dfrac{3}{n^{2}}+\dfrac{3}{b^{2}}-\dfrac{6}{nb}\right).
\end{align*}
Notice that 
\[\sum_{s=1}^{n-b}s=\dfrac{b(b-1)}{2}=\dfrac{b^{2}}{2}-\dfrac{b}{2},\]
\[\sum_{s=1}^{n-b}s^{2}=\dfrac{(b-1)b(2b-1)}{6}=\dfrac{b^{3}}{3}-\dfrac{b^{2}}{2}+\dfrac{b}{6}, \text{ and} \]
\[\sum_{s=1}^{n-b}s^{3}=\dfrac{(b-1)^{2}b^{2}}{4}=\dfrac{b^{2}}{4}-\dfrac{b^{3}}{2}+\dfrac{b^{2}}{4}.\]
Plugging into \eqref{eq:a2_f} yields 
\begin{align}\label{eq:Ea2_f}
E[a_2]&=\sum_{s=1}^{b-1}\sum_{l=0}^{n-b-s}\bigg[\dfrac{2}{b^{4}}s^{2}+\left(\dfrac{4}{nb^{2}}-\dfrac{4}{b^{3}}\right)s+\left(\dfrac{3}{n^{2}}+\dfrac{3}{b^{2}}-\dfrac{6}{nb}\right)\bigg]\nonumber\\
&=\sum_{s=1}^{b-1}\left[\dfrac{2}{b^{4}}s^{2}+\left(\dfrac{4}{nb^{2}}-\dfrac{4}{b^{3}}\right)s+\left(\dfrac{3}{n^{2}}+\dfrac{3}{b^{2}}-\dfrac{6}{nb}\right)\right]\left[n-b+1-s\right]\nonumber\\
&=\sum_{s=1}^{b-1}\bigg[-\dfrac{2}{b^{4}}s^{3}+\left(\dfrac{2n}{b^{4}}+\dfrac{2}{b^{3}}+\dfrac{2}{b^{4}}-\dfrac{4}{nb^{2}}\right)s^{2}+\left(\dfrac{5}{b^{2}}-\dfrac{4n}{b^{3}}+\dfrac{2}{bn}+\dfrac{4}{nb^{2}}-\dfrac{4}{b^{3}}-\dfrac{3}{n^{2}}\right)s\nonumber\\
&\ \ \ \ \ \ \  \ \ \ \ \ \  +\left(\dfrac{9}{n}+\dfrac{3n}{b^{2}}-\dfrac{9}{b}-\dfrac{3b}{n^{2}}+\dfrac{3}{n^{2}}+\dfrac{3}{b^{2}}-\dfrac{6}{nb}\right)\bigg]\nonumber\\
&=-\dfrac{2}{b^{4}}\left(\dfrac{b^{4}}{4}-\dfrac{b^{3}}{2}+\dfrac{b^{2}}{4}\right)+\left(\dfrac{2n}{b^{4}}+\dfrac{2}{b^{3}}+\dfrac{2}{b^{4}}-\dfrac{4}{nb^{2}}\right)\left(\dfrac{b^{3}}{3}-\dfrac{b^{2}}{2}+\dfrac{b}{6}\right)\nonumber\\
&\ \ \ \ \ +\left(\dfrac{5}{b^{2}}-\dfrac{4n}{b^{3}}+\dfrac{2}{bn}+\dfrac{4}{nb^{2}}-\dfrac{4}{b^{3}}-\dfrac{3}{n^{2}}\right)\left(\dfrac{b^{2}}{2}-\dfrac{b}{2}\right)\nonumber\\
&\ \ \ \ \ +\left(\dfrac{9}{n}+\dfrac{3n}{b^{2}}-\dfrac{9}{b}-\dfrac{3b}{n^{2}}+\dfrac{3}{n^{2}}+\dfrac{3}{b^{2}}-\dfrac{6}{nb}\right)\left(b-1\right)\nonumber\\
&=\dfrac{2n}{b^{4}}\cdot\dfrac{b^{3}}{3}-\dfrac{4n}{b^{3}}\cdot\dfrac{b^{2}}{2}+\dfrac{3n}{b^{2}}\cdot b +o\left(\dfrac{n}{b}\right)\nonumber\\
&=\dfrac{5}{3}\dfrac{n}{b}+o\left(\dfrac{n}{b}\right).
\end{align}
Similarly as $a_2$, we will calculate $E[a_3]$ at \eqref{eq:a3_f} by first calculating $E[Z_{1}^{2}Z_{2}^{2}]$ where $Z_{1}=(\dot{B}_{l}(b)-\bar{B})$, $Z_{2}=(\dot{B}_{l+s}(b)-\bar{B})$ for $l=0,...,(n-b-s)$ and $s=b,...(n-b)$. The joint distribution of $Z_{1}$ and $Z_{2}$ is 
\[\left[ \begin{matrix} Z_1 \\ Z_2 \end{matrix} \right] \sim N\left( \begin{bmatrix} 0\\ 0\end{bmatrix},\ \left[\begin{matrix}\dfrac{n-b}{bn}& -\dfrac{1}{n}\\ -\dfrac{1}{n}& \dfrac{n-b}{bn}\end{matrix}\right]\right),\]
resulting in
\[
Z_{1}|Z_{2}\sim N\left(\dfrac{-b}{n-b}Z_{2},\ \left[\dfrac{n-b}{bn}-\dfrac{bn}{n^{2}(n-b)}\right]\right) \text{ and }
\]
\[
Z_{2}\sim N\left(0,\ \dfrac{n-b}{bn}\right).
\]
Then
\begin{align}\label{eq:z1z2}
E[Z_{1}^{2}Z_{2}^{2}]&=E_{Z_{2}}[E_{Z_{1}|Z_{2}}[Z_{1}^{2}Z_{2}^{2}|Z_{2}]] \nonumber\\
&=E_{Z_{2}}[Z_{2}^{2}E_{Z_{1}|Z_{2}}[Z_{1}^{2}|Z_{2}]]\nonumber\\
&=E_{Z_{2}}\left[Z_{2}^{2}\left[\left(\dfrac{-b}{n-b}Z_{2}\right)^{2}+\left(\dfrac{n-b}{bn}-\dfrac{bn}{n^{2}(n-b)}\right)\right]\right]\nonumber\\
&=\left(\dfrac{-b}{n-b}\right)^{2}E_{Z_{2}}[Z_{2}^{4}]+\left(\dfrac{n-b}{bn}-\dfrac{bn}{n^{2}(n-b)}\right)E_{Z_{2}}[Z_{2}^{2}]\nonumber\\
&=\left(\dfrac{-b}{n-b}Z_{2}\right)^{2}3\left(\dfrac{n-b}{bn}\right)^{2}+\left(\dfrac{n-b}{bn}-\dfrac{bn}{n^{2}(n-b)}\right)\left(\dfrac{n-b}{bn}\right)\nonumber\\
&=\dfrac{2}{n^{2}}+\left(\dfrac{n-b}{bn}\right)^{2}\nonumber\\
&=\dfrac{3}{n^{2}}+\dfrac{1}{b^{2}}-\dfrac{2}{bn}.
\end{align}
Plugging into \eqref{eq:a3_f}, we have
\begin{align}\label{eq:Ea3_f}
E&[a_3] =\sum_{s=b}^{n-b}\sum_{l=0}^{n-b-s}\left(\dfrac{3}{n^{2}}+\dfrac{1}{b^{2}}-\dfrac{2}{bn}\right)\nonumber\\
&=\sum_{s=b}^{n-b}\left(\dfrac{3}{n^{2}}+\dfrac{1}{b^{2}}-\dfrac{2}{bn}\right)(n-b+1-s)\nonumber\\
&=\sum_{s=b}^{n-b}-\left(\dfrac{3}{n^{2}}+\dfrac{1}{b^{2}}-\dfrac{2}{bn}\right)s \nonumber\\
&\ \ \ \ \ +\left(\dfrac{n}{b^{2}}-\dfrac{3}{b}+\dfrac{5}{n}-\dfrac{3b^{2}}{n^{2}}+\dfrac{3}{n^{2}}+\dfrac{1}{b^{2}}-\dfrac{2}{bn}\right)\nonumber\\
&=-\left(\dfrac{3}{n^{2}}+\dfrac{1}{b^{2}}-\dfrac{2}{bn}\right)\left(\dfrac{n^{2}}{2}-bn+\dfrac{n}{2}\right)\nonumber\\
&\ \ \ \ \ +\left(\dfrac{n}{b^{2}}-\dfrac{3}{b}+\dfrac{5}{n}-\dfrac{3b^{2}}{n^{2}}+\dfrac{3}{n^{2}}+\dfrac{1}{b^{2}}-\dfrac{2}{bn}\right)\left(n-2b+1\right)\nonumber\\
&=-\left(\dfrac{1}{b^{2}}\cdot \dfrac{n^{2}}{2}-\dfrac{1}{b^{2}}\cdot(bn)-\dfrac{2}{bn}\dfrac{n^{2}}{2}\right)+\left(\dfrac{n}{b^{2}}\cdot n-\dfrac{n}{b^{2}}\cdot(2b)-\dfrac{3}{b}\cdot n\right)+o\left(\dfrac{n}{b}\right)\nonumber\\
&=\dfrac{n^{2}}{2b^{2}}-\dfrac{3n}{b}+o\left(\dfrac{n}{b}\right).
\end{align}
Combine \eqref{eq:Ea1_f}, \eqref{eq:Ea2_f} and \eqref{eq:Ea3_f}, \eqref{eq:A1_long_f} can be calculated by
\begin{equation}\label{eq:A1_result_f}
A_{1}=\dfrac{1}{(1-c)^{2}}\dfrac{b^{2}}{n^{2}}\left(E[a_1+2a_2+2a_3]\right)=\dfrac{1}{(1-c)^{2}}\left(1-\dfrac{8}{3}\dfrac{b}{n}\right)+o\left(\dfrac{b}{n}\right).
\end{equation}
Similarly, we can obtain $A_2$ at \eqref{eq:A2_f} by
\begin{equation}\label{eq:A2_result_f}
A_{2}=\left(\dfrac{c}{1-c}\right)^{2}\left(1-\dfrac{8}{3}\dfrac{cb}{n}\right)+o\left(\dfrac{b}{n}\right).
\end{equation}
To calculate $A_{3}$ at \eqref{eq:A3_f}, let $OL=[(\dot{B}_{p}(b)-\bar{B})^{2}(\dot{B}_{q}(cb)-\bar{B})^{2}]$ for $p$ and $q$ satisfying $q\geq p$ and $q+cb\leq p+b$, denote
\begin{equation}\label{eq:a4_f}
a_4=\sum_{s=1}^{cb-1}\sum_{l=0}^{n-b-s}(\dot{B}_{l}(b)-\bar{B})^{2}(\dot{B}_{l+(1-c)b+s}(cb)-\bar{B})^{2},
\end{equation}
and
\begin{equation}\label{eq:a5_f}
a_5=\sum_{s=cb}^{n-b}\sum_{l=0}^{n-b-s}(\dot{B}_{l}(b)-\bar{B})^{2}(\dot{B}_{l+(1-c)b+s}(cb)-\bar{B})^{2}.
\end{equation}
Then $A_{3}$ has the following expression:
\begin{align}\label{eq:A3_long_f}
A_{3}&=-\dfrac{2c^{2}}{(1-c)^{2}}\dfrac{b^{2}}{n^{2}}E\left[\sum_{l=0}^{n-b}(\dot{B}_{l}(b)-\bar{B})^{2}\right]\left[\sum_{l=0}^{n-cb}(\dot{B}_{l}(cb)-\bar{B})^{2}\right]\nonumber\\
&=-\dfrac{2c^{2}}{(1-c)^{2}}\dfrac{b^{2}}{n^{2}}E[((1-c)b+1)(n-b+1)\cdot OL\nonumber\\
&\ \ \ \ \ +2\sum_{s=1}^{cb-1}\sum_{l=0}^{n-b-s}(\bar{B}_{l}(b)-\bar{B})^{2}(\bar{B}_{l+(1-c)b+s}(cb)-\bar{B})^{2}\nonumber\\
&\ \ \ \ \ +2\sum_{s=cb}^{n-b}\sum_{l=0}^{n-b-s}(\bar{B}_{l}(b)-\bar{B})^{2}(\bar{B}_{l+(1-c)b+s}(cb)-\bar{B})^{2}]\nonumber\\
&=-\dfrac{2c^{2}}{(1-c)^{2}}\dfrac{b^{2}}{n^{2}}E[((1-c)b+1)(n-b+1)\cdot OL+2a_4+2a_5]
\end{align}
First consider $E[OL]$. Denote $Z_{1}=(\dot{B}_{p}(b)-\bar{B})$ and $Z_{2}=(\dot{B}_{q}(cb)-\bar{B})$ for $p$ and $q$ satisfying $q\geq p$ and $q+cb\leq p+b$, then
\[\left[ \begin{matrix} Z_1 \\ Z_2 \end{matrix} \right] \sim N\left( \begin{bmatrix} 0\\ 0\end{bmatrix},\ \left[\begin{matrix}\dfrac{n-b}{bn}& \dfrac{n-b}{bn}\\ \dfrac{n-b}{bn}& \dfrac{n-cb}{cbn}\end{matrix}\right]\right),\]
resulting in
\[
Z_{1}|Z_{2}\sim N\left(\dfrac{c(n-b)}{n-cb}Z_{2},\ \dfrac{(1-c)n+(c-1)b}{b(n-cb)}\right) \text{ and}
\]
\[
Z_{2}\sim N\left(0,\ \dfrac{n-cb}{cbn}\right).
\]
Then
\begin{align}\label{eq:z1z2_0}
E[Z_{1}^{2}Z_{2}^{2}]&=E_{Z_{2}}[E_{Z_{1}|Z_{2}}[Z_{1}^{2}Z_{2}^{2}|Z_{2}]]\nonumber\\
&=E_{Z_{2}}[Z_{2}^{2}E_{Z_{1}|Z_{2}}[Z_{1}^{2}|Z_{2}]]\nonumber\\
&=E_{Z_{2}}\left[Z_{2}^{2}\left[\left(\dfrac{c(n-b)}{n-cb}Z_{2}\right)^{2}+\dfrac{(1-c)n+(c-1)b}{b(n-cb)}\right]\right]\nonumber\\
&=\left(\dfrac{c(n-b)}{n-cb}Z_{2}\right)^{2}E_{Z_{2}}[Z_{2}^{4}]+\dfrac{(1-c)n+(c-1)b}{b(n-cb)}E_{Z_{2}}[Z_{2}^{2}]\nonumber\\
&=\left(\dfrac{c(n-b)}{n-cb}Z_{2}\right)^{2}3\left(\dfrac{n-cb}{cbn}\right)^{2}+\dfrac{(1-c)n+(c-1)b}{b(n-cb)}\left(\dfrac{n-cb}{cbn}\right)\nonumber\\
&=\dfrac{2c+1}{c}\dfrac{1}{b^2}+\dfrac{3}{n^{2}}-\dfrac{5c+1}{c}\dfrac{1}{bn},
\end{align}
thus
\begin{align}\label{eq:EOL_f}
E[((1-c) & b+1)(n-b+1)\cdot OL]\nonumber\\
&=((1-c)b+1)(n-b+1)\left(\dfrac{2c+1}{c}\dfrac{1}{b^2}+\dfrac{3}{n^{2}}-\dfrac{5c+1}{c}\dfrac{1}{bn}\right)\nonumber\\
&=(1-c)bn\dfrac{2c+1}{c}\dfrac{1}{b^{2}}\nonumber\\
&=\dfrac{(2c+1)(1-c)}{c}\dfrac{n}{b}.
\end{align}
To calculate $E[a_4]$ at \eqref{eq:a4_f}, define $Z_{1}=(\dot{B}_{l}(b)-\bar{B})$ , $Z_{2}=(\dot{B}_{l+(1-c)b+s}(cb)-\bar{B})$. For $l=0,...,(n-b-s)$ and $s=1,...(cb-1)$
\[\left[ \begin{matrix} Z_1 \\ Z_2 \end{matrix} \right] \sim N\left( \begin{bmatrix} 0\\ 0\end{bmatrix},\ \left[\begin{matrix}\dfrac{n-b}{bn}& \dfrac{cbn-cb^{2}-sn}{cb^{2}n}\\ \dfrac{cbn-cb^{2}-sn}{cb^{2}n}& \dfrac{n-cb}{cbn}\end{matrix}\right]\right),\]
resulting in
\[
Z_{1}|Z_{2}\sim N\left(\dfrac{cbn-cb^{2}-sn}{b(n-cb)}Z_{2},\ \dfrac{(c-c^{2})b^{2}(n-b)-s^{2}n+2cb(n-b)s}{cb^{3}(n-cb)}\right) \text{ and}
\]
\[
Z_{2}\sim N\left(0,\ \dfrac{n-cb}{cbn}\right).
\]
Therefore 
\begin{align*}
&\ \ \ \ \ \ E[Z_{1}^{2}Z_{2}^{2}]\\
&=\left(\dfrac{cbn-cb^{2}-sn}{b(n-cb)}\right)^{2}E_{Z_{2}}[Z_{2}^{4}]+\dfrac{(c-c^{2})b^{2}(n-b)-s^{2}n+2cb(n-b)s}{cb^{3}(n-cb)}E_{Z_{2}}[Z_{2}^{2}]\\
&=\left(\dfrac{cbn-cb^{2}-sn}{b(n-cb)}\right)^{2}3\left(\dfrac{n-cb}{cbn}\right)^{2}+\dfrac{(c-c^{2})b^{2}(n-b)-s^{2}n+2cb(n-b)s}{cb^{3}(n-cb)}\left(\dfrac{n-cb}{cbn}\right)\\
&=\dfrac{2}{c^{2}b^{4}}s^{2}+\left(\dfrac{4}{cb^{2}n}-\dfrac{4}{cb^{3}}\right)s+\left(\dfrac{1+2c}{cb^{2}}+\dfrac{3}{n^2}-\dfrac{1+5c}{cbn}\right).
\end{align*}
Notice
$$\sum_{s=1}^{cb-1}s=\dfrac{c^{2}b^{2}}{2}-\dfrac{cb}{2},$$
$$\sum_{s=1}^{cb-1}s^{2}=\dfrac{1}{6}(cb-1)(cb)(2cb-1)=\dfrac{c^{3}b^{3}}{3}-\dfrac{c^{2}b^{2}}{2}+\dfrac{cb}{6},$$
$$\sum_{s=1}^{cb-1}s^{3}=\dfrac{(cb)^{2}(cb-1)^{2}}{4}=\dfrac{c^{4}b^{4}}{4}-\dfrac{c^{3}b^{3}}{2}+\dfrac{c^{2}b^{2}}{4}.$$
Then $E[a_4]$ at \eqref{eq:a4_f} can be approached by
\begin{align}\label{Ea4_f}
E[a_4]&=E\left[\sum_{s=1}^{cb-1}\sum_{l=0}^{n-b-s}(\dot{B}_{l}(b)-\bar{B})^{2}(\dot{B}_{l+(1-c)b+s}(cb)-\bar{B})^{2}\right]\nonumber\\
&=\sum_{s=1}^{cb-1}\sum_{l=0}^{n-b-s}\left[\dfrac{2}{c^{2}b^{4}}s^{2}+\left(\dfrac{4}{cb^{2n}}-\dfrac{4}{cb^{3}}\right)s+\left(\dfrac{1+2c}{cb^{2}}+\dfrac{3}{n}-\dfrac{1+5c}{cbn}\right)\right]\nonumber\\
&=\sum_{s=1}^{cb-1}\bigg[-\dfrac{2}{c^{2}b^{4}}s^{3}+\left(\dfrac{2(n-b+1)}{c^{2}b^{4}}-\dfrac{4}{cb^{2}n}+\dfrac{4}{cb^{3}}\right)s^{2}\nonumber\\
&\ \ \ \ \ \ \ \ \ \ \ \ \ \ \ \  +\left(\left(\dfrac{4}{cb^{2}n}-\dfrac{4}{cb^{3}}\right)(n-b+1)-\left(\dfrac{1+2c}{cb^{2}}+\dfrac{3}{n^{2}}-\dfrac{1+5c}{cbn}\right)\right)s\nonumber\\
&\ \ \ \ \ \ \ \ \ \ \ \ \ \ \ \ +\left(\dfrac{1+2c}{cb^{2}}+\dfrac{3}{n^{2}}-\dfrac{1+5c}{cbn}\right)(n-b+1)\bigg]\nonumber\\
&=-\dfrac{2}{c^{2}b^{4}}\left(\dfrac{c^{4}b^{4}}{4}-\dfrac{c^{3}b^{3}}{2}+\dfrac{c^{2}b^{2}}{4}\right)\nonumber\\
&\ \ \ \ \ \ \ \ \ +\left(\dfrac{2(n-b+1)}{c^{2}b^{4}}-\dfrac{4}{cb^{2}n}+\dfrac{4}{cb^{3}}\right)\left(\dfrac{c^{3}b^{3}}{3}-\dfrac{c^{2}b^{2}}{2}+\dfrac{cb}{6}\right)\nonumber\\
&\ \ \ \ \ \ \ \ \ +\left(\left(\dfrac{4}{cb^{2}n}-\dfrac{4}{cb^{3}}\right)(n-b+1)-\left(\dfrac{1+2c}{cb^{2}}+\dfrac{3}{n^{2}}-\dfrac{1+5c}{cbn}\right)\right)\left(\dfrac{c^{2}b^{2}}{2}-\dfrac{cb}{2}\right)\nonumber\\
&\ \ \ \ \ \ \ \ \ +\left(\left(\dfrac{1+2c}{cb^{2}}+\dfrac{3}{n^{2}}-\dfrac{1+5c}{cbn}\right)(n-b+1)(cb-1)\right)\nonumber\\
&=\left[\dfrac{2n}{c^{2}b^{4}}\cdot\dfrac{c^{3}b^{3}}{3}-\dfrac{4n}{cb^{3}}\cdot\dfrac{c^{2}b^{2}}{2}+\dfrac{(1+2c)n}{cb^{2}}(cb)\right]+o\left(\dfrac{n}{b}\right)\nonumber\\
&=\left(\dfrac{2c}{3}+1\right)\dfrac{n}{b}+o\left(\dfrac{n}{b}\right).
\end{align}
Next calculate $E[a_5]$ at \eqref{eq:a5_f}. First consider the joint distribution of $Z_{1}=(\dot{B}_{l}(b)-\bar{B})$ and $Z_{2}=(\dot{B}_{l+(1-c)b+s}(cb)-\bar{B})$ for $l=0,...,(n-b-s)$ and $s=cb,...(n-cb).$
$$\left[ \begin{matrix} Z_1 \\ Z_2 \end{matrix} \right] \sim N\left( \begin{bmatrix} 0\\ 0\end{bmatrix},\ \left[\begin{matrix}\dfrac{n-b}{bn}& -\dfrac{1}{n}\\ -\dfrac{1}{n}& \dfrac{n-cb}{cbn}\end{matrix}\right]\right),$$
resulting in
\[
Z_{1}|Z_{2}\sim N\left(\dfrac{cb}{cb-n}Z_{2},\ \bigg(\dfrac{n-b}{bn}-\dfrac{cb}{n(n-cb)}\bigg)\right) \text{ and}
\]
\[
Z_{2}\sim N\left(0,\ \dfrac{n-cb}{cbn}\right).
\]
Then
\begin{align}\label{eq:z1z2_1}
&\ \ \ \ \ \ E[Z_{1}^{2}Z_{2}^{2}]\nonumber\\
&=\left(\dfrac{cb}{cb-n}\right)^{2}E_{Z_{2}}[Z_{2}^{4}]+\left(\dfrac{n-b}{bn}-\dfrac{cb}{n(n-cb)}\right)E_{Z_{2}}[Z_{2}^{2}]\nonumber\\
&=\left(\dfrac{cb}{cb-n}\right)^{2}3\left(\dfrac{n-cb}{cbn}\right)^{2}+\bigg(\dfrac{n-b}{bn}-\dfrac{cb}{n(n-cb)}\bigg)\left(\dfrac{n-cb}{cbn}\right)\nonumber\\
&=\dfrac{3}{n^{2}}+\dfrac{1}{cb^{2}}-\dfrac{c+1}{c}\dfrac{1}{bn}.
\end{align}
Plug in \eqref{eq:a5_f},
\begin{align}\label{Ea5_f}
E[a_5]&=E\left[\sum_{s=cb}^{n-b}\sum_{l=0}^{n-b-s}(\dot{B}_{l}(b)-\bar{B})^{2}(\dot{B}_{l+(1-c)b+s}(cb)-\bar{B})^{2}\right]\nonumber\\
&=\sum_{s=cb}^{n-b}\sum_{l=0}^{n-b-s}\left(\dfrac{3}{n^{2}}+\dfrac{1}{cb^{2}}-\dfrac{c+1}{c}\dfrac{1}{bn}\right)\nonumber\\
&=\sum_{s=cb}^{n-b}\left(\dfrac{3}{n^{2}}+\dfrac{1}{cb^{2}}-\dfrac{c+1}{c}\dfrac{1}{bn}\right)(n-b+1-s)\nonumber\\
&=\sum_{s=cb}^{n-b}\bigg[\left(\dfrac{3}{n^{2}}+\dfrac{1}{cb^{2}}-\dfrac{c+1}{c}\dfrac{1}{bn}\right)(n-b+1)-\left(\dfrac{3}{n^{2}}+\dfrac{1}{cb^{2}}-\dfrac{c+1}{c}\dfrac{1}{bn}\right)s\bigg]\nonumber\\
&=\left(\dfrac{3}{n^{2}}+\dfrac{1}{cb^{2}}-\dfrac{c+1}{c}\dfrac{1}{bn}\right)(n-b+1)(n-(1+c)b+1)\nonumber\\
&\ \ \ \ \ -\left(\dfrac{3}{n^{2}}+\dfrac{1}{cb^{2}}-\dfrac{c+1}{c}\dfrac{1}{bn}\right)\left(\dfrac{n^{2}}{2}-bn\right)\nonumber\\
&=\left(\dfrac{3}{n}+\dfrac{n}{cb^{2}}-\dfrac{c+2}{c}\dfrac{1}{b}-\dfrac{3b}{n^{2}}+\dfrac{c+1}{c}\dfrac{1}{n}+\dfrac{3}{n^{2}}+\dfrac{1}{cb^{2}}-\dfrac{c+1}{c}\dfrac{1}{bn}\right)(n-(1+c)b+1)\nonumber\\
&\ \ \ \ \ -\left(\dfrac{3}{n^{2}}+\dfrac{1}{cb^{2}}-\dfrac{c+1}{c}\dfrac{1}{bn}\right)\left(\dfrac{n^{2}}{2}-bn\right)\nonumber\\
&=\left(-\dfrac{c+2}{c}\dfrac{1}{b}\cdot n+\dfrac{n}{cb^{2}}\cdot n-\dfrac{n}{cb^{2}}\cdot (1+c)b\right)\nonumber\\
&\ \ \ \ \ -\left(\dfrac{1}{cb^{2}}\cdot \dfrac{n^{2}}{2}-\dfrac{1}{cb^{2}}\cdot bn-\dfrac{c+1}{c}\dfrac{1}{bn}\cdot\dfrac{n^{2}}{2}\right)+o\left(\dfrac{n}{b}\right)\nonumber\\
&=\dfrac{1}{2c}\dfrac{n^{2}}{b^{2}}-\left(\dfrac{3}{2}+\dfrac{3}{2c}\right)\dfrac{n}{b}+o\left(\dfrac{n}{b}\right).
\end{align}
Combine \eqref{eq:EOL_f}, \eqref{Ea4_f} and \eqref{Ea5_f}, then $A_3$ at \eqref{eq:A3_long_f} can be calculated by
\begin{align}\label{A3_result_f}
A_{3}&=-\dfrac{2c^{2}}{(1-c)^{2}}\dfrac{b^{2}}{n^{2}} (E[((1-c)b+1)(n-b+1)\cdot OL]+2ED+2EG)\nonumber\\
&=-\dfrac{2c^{2}}{(1-c)^{2}}\dfrac{b^{2}}{n^{2}}\left[\dfrac{(2c+1)(1-c)}{c}\dfrac{n}{b}+\left(\dfrac{4c}{3}+2\right)\dfrac{n}{b}-\left(3+\dfrac{3}{c}\right)\dfrac{n}{b}+\dfrac{1}{c}\dfrac{n^{2}}{b^{2}}\right]+o\left(\dfrac{b}{n}\right)\nonumber\\
&=-\dfrac{2c^{2}}{(1-c)^{2}}\left(\dfrac{1}{c}-\dfrac{2c^{2}+6}{3c}\dfrac{b}{n}\right)+o\left(\dfrac{b}{n}\right).
\end{align}
From $A_{1}$, $A_{2}$, $A_{3}$ at \eqref{eq:A1_result_f}, \eqref{eq:A2_result_f}, \eqref{A3_result_f}, $E[\tilde{\sigma}^{4}_{ft}]$ at \eqref{eq:A1+A2+A3_f} becomes 
\begin{align}\label{eq:E4(A1+A2+A3)}
E[\tilde{\sigma}^{4}_{ft}]&=A_{1}+A_{2}+A_{3}\nonumber\\
&=\dfrac{1}{(1-c)^{2}}\left(1-\dfrac{8}{3}\dfrac{b}{n}\right)+\dfrac{c^{2}}{(1-c)^{2}}\left(1-\dfrac{8c}{3}\dfrac{b}{n}\right)-\dfrac{2c^{2}}{(1-c)^{2}}\left(\dfrac{1}{c}-\dfrac{2c^{2}+6}{3c}\dfrac{b}{n}\right)\nonumber\\
&=1+\dfrac{-4c^{3}+12c-8}{3(1-c)^{2}}\dfrac{b}{n}+o\left(\dfrac{b}{n}\right).
\end{align}
We also need $(E[\tilde{\sigma}^{2}_{ft}])^{2}$ to calculate $Var[\tilde{\sigma}^{2}_{ft}]$. By \eqref{eq:E(B)^2=1},
\[
E \left[ \dfrac{bn}{n-b}(\dot{B}_{l}(b)-\bar{B})^{2} \right]=1.
\]
Therefore
\[
E \left[ \sum_{l=0}^{n-b}(\dot{B}_{l}(b)-\bar{B})^{2} \right]=\dfrac{(n-b)(n-b+1)}{bn}
\]
and
\[
E \left[ \sum_{l=0}^{n-cb}(\dot{B}_{l}(cb)-\bar{B})^{2} \right]=\dfrac{(n-cb)(n-cb+1)}{cbn}.
\]
Then
\begin{align}\label{eq:E2(A1+A2+A3)}
(E[\tilde{\sigma}^{2}_{ft}])^{2}&=\left(E\left[\dfrac{1}{1-c}\dfrac{b}{n}\sum_{l=0}^{n-b}(\dot{B}_{l}(b)-\bar{B})^{2}-\dfrac{c}{1-c}\dfrac{cb}{n}\sum_{l=0}^{n-cb}(\dot{B}_{l}(cb)-\bar{B})^{2}\right]\right)^{2}\nonumber\\
&=\left[\dfrac{1}{1-c}\dfrac{b}{n}\dfrac{(n-b)(n-b+1)}{bn}-\dfrac{c}{1-c}\dfrac{cb}{n}\dfrac{(n-cb)(n-cb+1)}{cbn}\right]^{2}\nonumber\\
&=\left[\dfrac{(n-b)(n-b+1)}{(1-c)n^{2}}-\dfrac{c(n-cb)(n-cb+1)}{(1-c)n^{2}}\right]^{2}\nonumber\\
&=\dfrac{1}{(1-c)^{2}n^{4}}[(n^{2}+b^{2}-2bn)(n^{2}+b^{2}-2bn+1+2n-2b)]\nonumber\\
&\ \ \ \ \ +\dfrac{c^{2}}{(1-c)^{2}n^{4}}[(n^{2}+c^{2}b^{2}-2cbn)(n^{2}+c^{2}b^{2}-2cbn+1+2n-2cb)]\nonumber\\
&\ \ \ \ \ -\dfrac{2c}{(1-c)^{2}n^{4}}[(n^{2}+b^{2}-2bn+n-b)(n^{2}+c^{2}b^{2}-2cbn+n-cb)]\nonumber\\
&=\dfrac{1}{(1-c)^{2}n^{4}}[n^{4}-4bn^{3}]+\dfrac{c^{2}}{(1-c)^{2}n^{4}}[n^{4}-4cbn^{3}]\nonumber\\
&\ \ \ \ \  -\dfrac{2c}{(1-c)^{2}n^{4}}[n^{4}-2cbn^{3}-2bn^{3}]+o\left(\dfrac{b}{n}\right)\nonumber\\
&=\dfrac{1}{(1-c)^{2}n^{4}}[(1-c)^{2}n^{4}+(4c^{2}+4c-4c^{3}-4)bn^{3}]+o\left(\dfrac{b}{n}\right)\nonumber\\
&=1+\dfrac{4c^{2}+4c-4c^{3}-4}{(1-c)^{2}}\dfrac{b}{n}+o\left(\dfrac{b}{n}\right).
\end{align}
Combine \eqref{eq:E4(A1+A2+A3)} and \eqref{eq:E2(A1+A2+A3)} to get
\begin{align*}
Var[\tilde{\sigma}^{2}_{ft}]&=E[\tilde{\sigma}^{4}_{ft}]-(E[\tilde{\sigma}^{2}_{ft}])^{2}\\
&=\left(1+\dfrac{-4c^{3}+12c-8}{3(1-c)^{2}}\dfrac{b}{n}\right)-\left(1+\dfrac{4c^{2}+4c-4c^{3}-4}{(1-c)^{2}}\dfrac{b}{n}\right)+o\left(\dfrac{b}{n}\right)\\
&=\dfrac{8c^{3}-12c^{2}+4}{3(1-c)^{2}}\dfrac{b}{n}+o\left(\dfrac{b}{n}\right)\\
&=\dfrac{8c[(1-c)^{2}+\dfrac{1}{2c}(c-1)^{2}]}{3(1-c)^{2}}+o\left(\dfrac{b}{n}\right)\\
&=\left(\dfrac{8}{3}c+\dfrac{4}{3}\right)\dfrac{b}{n}+o\left(\dfrac{b}{n}\right).
\end{align*}
\end{proof}


Let $\tilde{\sigma}^2_f$ be the Brownian motion expression of $\hat{\sigma}^2_f$. Corollary~\ref{cor:lemmaf} below follows Lemma~\ref{lemma:f} by letting $c=1/2$. 
\begin{corollary}\label{cor:lemmaf}
Under Assumption \ref{ass:batch},
\[\dfrac{n}{b}Var[\tilde{\sigma}^{2}_{f}]=\dfrac{8}{3}+o(1).\]
\end{corollary}

Recall $\hat{\sigma}^2_w$ denotes the univariate weighted BM estimator with the Bartlett flat top lag window at \eqref{eq:ft}.   Suppose $n=ab$, then consider the corresponding Brownian motion expression
\[
\tilde{\sigma}^{2}_{w}=\dfrac{2b}{a-1}\sum_{l=0}^{a-1}(\bar{B}_{l}(b)-\bar{B})^{2}-\dfrac{b/2}{2a-1}\sum_{l=0}^{2a-1}(\bar{B}_{l}(b/2)-\bar{B})^{2}.
\]
\begin{lemma}\label{lemma:w}
Under Assumption \ref{ass:batch}
\[\dfrac{n}{b}Var[\tilde{\sigma}^{2}_{w}]=5+o(1).\]
\end{lemma}

\begin{proof}
Since $Var[\tilde{\sigma}^{2}_{w}]=E[\tilde{\sigma}^{4}_{w}]-(E[\tilde{\sigma}^{2}_{w}])^{2}$, first consider $E[\tilde{\sigma}^{4}_{w}]$.
\begin{align}\label{eq:A1+A2+A3_w}
E[\tilde{\sigma}^{4}_{w}]& =E\left[\Big(\dfrac{2b}{a-1}\sum_{l=0}^{a-1}(\bar{B}_{l}(b)-\bar{B})^{2}-\dfrac{b/2}{2a-1}\sum_{l=0}^{2a-1}(\bar{B}_{l}(b/2)-\bar{B})^{2}\Big)^{2}\right]\nonumber\\
&=E\left[\Big(\dfrac{2b}{a-1}\Big)^{2}\Big(\sum_{l=0}^{a-1}(\bar{B}_{l}(b)-\bar{B})^{2}\Big)^{2}+\Big(\dfrac{b}{4a-2}\Big)^{2}\Big(\sum_{l=0}^{2a-1}(\bar{B}_{l}(b/2)-\bar{B})^{2}\Big)^{2}\right.\nonumber\\
&\left.-\dfrac{4b^2}{(a-1)(4a-2)}\Big(\sum_{l=0}^{a-1}(\bar{B}_{l}(b)-\bar{B})^{2}\Big)\Big(\sum_{l=0}^{2a-1}(\bar{B}_{l}(b/2)-\bar{B})^{2}\Big)\right]\nonumber\\
&=A_{1}+A_{2}+A_{3}
\end{align}
where 
\begin{equation}\label{A1_w}
A_{1}=E\left[\Big(\dfrac{2b}{a-1}\Big)^{2}\Big(\sum_{l=0}^{a-1}(\bar{B}_{l}(b)-\bar{B})^{2}\Big)^{2}\right],
\end{equation}
\begin{equation}\label{A2_w}
A_{2}=E\left[\Big(\dfrac{b}{4a-2}\Big)^{2}\Big(\sum_{l=0}^{2a-1}(\bar{B}_{l}(b/2)-\bar{B})^{2}\Big)^{2}\right], \text{ and}
\end{equation}
\begin{equation}\label{A3_w}
A_{3}=E\left[-\dfrac{4b^2}{(a-1)(4a-2)}\Big(\sum_{l=0}^{a-1}(\bar{B}_{l}(b)-\bar{B})^{2}\Big)\Big(\sum_{l=0}^{2a-1}(\bar{B}_{l}(b/2)-\bar{B})^{2}\Big)\right].
\end{equation}
Denote
\[a_1=\sum_{l=0}^{a-1}(\bar{B}_{l}(b)-\bar{B})^{4}\]
and 
\[a_2=\sum_{s=1}^{a-1}\sum_{l=0}^{a-1-s}(\bar{B}_{l}(b)-\bar{B})^{2}(\bar{B}_{l+s}(b)-\bar{B})^{2}.\]
Then \eqref{A1_w} can be expressed as
\begin{align}\label{eq:A1_long_w}
A_{1}&=\Big(\dfrac{2b}{a-1}\Big)^2E\Bigg[\sum_{l=0}^{a-1}(\bar{B}_{l}(b)-\bar{B})^{4}+2\sum_{s=1}^{a-1}\sum_{l=0}^{a-1-s}(\bar{B}_{l}(b)-\bar{B})^{2}(\bar{B}_{l+s}(b)-\bar{B})^{2}\Bigg]\nonumber\\
&=\Big(\dfrac{2b}{a-1}\Big)^2E[a_1+2a_2].
\end{align}
First consider $E[a_1]$. By \eqref{eq:E(B)^4},
\[E[(\bar{B}_l(b)-\bar{B})^4]=3\Big(\dfrac{n-b}{bn}\Big)^2,\]
hence 
\begin{align}\label{eq:Ea1_w}
E[a_1]&=E\left[\sum_{l=0}^{a-1}(\bar{B}_{l}(b)-\bar{B})^{4}\right]\nonumber\\
&=3a\left(\dfrac{n-b}{bn}\right)^2\nonumber\\
&=3a\left(\dfrac{1}{b^2}+\dfrac{1}{n^2}-\dfrac{2}{bn}\right)\nonumber\\
&=3a\left(\dfrac{1}{b^2}+\dfrac{1}{a^2 b^2}-\dfrac{2}{ab^2}\right)\nonumber\\
&=\dfrac{3a}{b^2}+o\left(\dfrac{a}{b^2}\right).
\end{align}
From \eqref{eq:z1z2}, 
\[E[(\bar{B}_l(b)-\bar{B})^2(\bar{B}_{l+s}(b)-\bar{B})^2]=\dfrac{3}{n^2}+\dfrac{1}{b^2}-\dfrac{2}{bn}\]
and therefore
\begin{align}\label{eq:Ea2_w}
E[a_2]&=E\left[\sum_{s=1}^{a-1}\sum_{l=0}^{a-1-s}(\bar{B}_{l}(b)-\bar{B})^{2}(\bar{B}_{l+s}(b)-\bar{B})^{2}\right]\nonumber\\
&=\left(\dfrac{3}{n^2}+\dfrac{1}{b^2}-\dfrac{2}{bn}\right)\sum_{s=1}^{a-1}\sum_{l=0}^{a-1-s}1\nonumber\\
&=\left(\dfrac{3}{n^2}+\dfrac{1}{b^2}-\dfrac{2}{bn}\right)\dfrac{a(a-1)}{2}\nonumber\\
&=\dfrac{1}{2}\left(\dfrac{a^2}{b^2}-\dfrac{3a}{b^2}\right)+o\left(\dfrac{a}{b^2}\right).
\end{align}
Plug $E[a_1]$, $E[a_2]$ at \eqref{eq:Ea1_w} and \eqref{eq:Ea2_w} in \eqref{eq:A1_long_w}, we have
\begin{align}\label{eq:A1_result_w}
A_1&=\left(\dfrac{2b}{a-1}\right)^2(EA+2EB)\nonumber\\
&=\dfrac{4b^2}{(a-1)^2}\left[\dfrac{3a}{b^2}+\dfrac{a^2}{b^2}-\dfrac{3a}{b^2}+o\left(\dfrac{a}{b^2}\right)\right]\nonumber\\
&=\dfrac{4a^2}{(a-1)^2}+o\left(\dfrac{1}{a}\right).
\end{align}
Next consider $A_2$ at \eqref{A2_w}. Denote
\[a_3=\sum_{l=0}^{2a-1}(\bar{B}_{l}(b/2)-\bar{B})^{4}\]
and
\[a_4=\sum_{s=1}^{2a-1}\sum_{l=0}^{2a-1-s}(\bar{B}_{l}(b/2)-\bar{B})^{2}(\bar{B}_{l+s}(b/2)-\bar{B})^{2},\]
then
\begin{align}\label{eq:A2_long_w}
A_2&=\dfrac{b^2}{(4a-2)^2}E\left[\sum_{l=0}^{2a-1}\left(\bar{B}_{l}(b/2)-\bar{B}\right)^4 \right. \notag \\
& \quad \quad \quad \left. + 2 \sum_{s=1}^{2a-1}\sum_{l=0}^{2a-1-s}\left(\bar{B}_l(b/2)-\bar{B}\right)^2\left(\bar{B}_{l+s}(b/2)-\bar{B}\right)^2\right]\nonumber\\
&=\dfrac{b^2}{(4a-2)^2}E[a_3+2a_4.]
\end{align}
By \eqref{eq:E(B)^4}
\[E[(\bar{B}_l(b/2)-\bar{B})^4]=3\Big(\dfrac{2n-b}{bn}\Big)^2,\]
hence 
\begin{align}\label{eq:Ea3_w}
E[a_3]&=E\left[\sum_{l=0}^{a-1}(\bar{B}_{l}(b/2)-\bar{B})^{4}\right]\nonumber\\
&=6a\left(\dfrac{2n-b}{bn}\right)^2\nonumber\\
&=6a\left(\dfrac{4}{b^2}+\dfrac{1}{n^2}-\dfrac{4}{bn}\right)\nonumber\\
&=6a\left(\dfrac{4}{b^2}+\dfrac{1}{a^2 b^2}-\dfrac{4}{ab^2}\right)\nonumber\\
&=\dfrac{24a}{b^2}+o\left(\dfrac{a}{b^2}\right).
\end{align}
To calculate $E[a_4]$, define $Z_{1}=(\bar{B}_{l}(b/2)-\bar{B})$ and $Z_{2}=(\bar{B}_{l+s}(b/2)-\bar{B})$ for $l=0,...,(2a-1-s)$ and $s=1,...,(2a-1).$ Consider the joint distribution of $Z=(Z_{1},Z_{2})^{T}.$ Denote $U=(U_{1},...U_{n})^{T}$, $Z_{1}$ and $Z_{2}$ are linear combinations of i.i.d. standard normal variables, then $Z=VU$ where
\[V = \begin{bmatrix} -\dfrac{1}{n}&\cdots&-\dfrac{1}{n} &\dfrac{2n-b}{bn}&\dfrac{2n-b}{bn}&-\dfrac{1}{n}&\cdots&\cdots&-\dfrac{1}{n}  \\
\dfrac{2n-b}{bn}&\cdots&\dfrac{2n-b}{bn}&-\dfrac{1}{n}&\cdots&\cdots&\cdots&\cdots&-\dfrac{1}{n}\end{bmatrix}. \]
The joint distribution of $Z$ is 
\[\left[ \begin{matrix} Z_1 \\ Z_2 \end{matrix} \right] \sim N\left( \begin{bmatrix} 0\\ 0\end{bmatrix},\ \left[\begin{matrix}\dfrac{2n-b}{bn}& -\dfrac{1}{n}\\ -\dfrac{1}{n}& \dfrac{2n-b}{bn}\end{matrix}\right]\right)\] .
The conditional distribution of $Z_{1}|Z_{2}$ and the marginal distribution of $Z_{2}$ are
\[
Z_{1}|Z_{2}\sim N\left(\dfrac{b}{b-2n} Z_{2},\ \dfrac{4n-4b}{2bn-b^2}\right) \text{ and}
\]
\[
Z_{2}\sim N\left(0,\ \dfrac{2n-b}{bn}\right).
\]
Now we have the expectation
\begin{align*}
E[Z_{1}^{2}Z_{2}^{2}]&=E_{Z_{2}}[E_{Z_{1}|Z_{2}}[Z_{1}^{2}Z_{2}^{2}|Z_{2}]]\\
&=E_{Z_{2}}[Z_{2}^{2}E_{Z_{1}|Z_{2}}[Z_{1}^{2}|Z_{2}]]\\
&=E_{Z_{2}}\left[Z_{2}^{2}\left(\left(\dfrac{b}{b-2n}Z_{2}\right)^{2}+\dfrac{4n-4b}{2bn-b^2}\right)\right]\\
&=\left(\dfrac{b}{b-2n}\right)^{2}E_{Z_{2}}[Z_{2}^{4}]+\dfrac{4n-4b}{2bn-b^2}E_{Z_{2}}[Z_{2}^{2}]\\
&=\left(\dfrac{b}{b-2n}\right)^{2}3\left(\dfrac{2n-b}{bn}\right)^{2}+\dfrac{4n-4b}{2bn-b^2}\left(\dfrac{2n-b}{bn}\right)\\
&=\dfrac{3}{n^2}+\dfrac{4(n-b)}{b^2n}\\
&=\dfrac{3}{n^2}+\dfrac{4}{b^2}-\dfrac{4}{bn}.
\end{align*}
Then
\begin{align}\label{eq:Ea4_w}
E[a_4]&=\left(\dfrac{3}{n^2}+\dfrac{4}{b^2}-\dfrac{4}{bn}\right)\sum_{s=1}^{2a-1}\sum_{l=0}^{2a-1-s}1\nonumber\\
&=\left(\dfrac{3}{n^2}+\dfrac{4}{b^2}-\dfrac{4}{bn}\right)\dfrac{2a(2a-1)}{2}\nonumber\\
&=\dfrac{8a^2}{b^2}-\dfrac{12a}{b^2}+o\left(\dfrac{a}{b^2}\right).
\end{align}
Plug $E[a_3]$, $E[a_4]$ at \eqref{eq:Ea3_w}, \eqref{eq:Ea4_w} in \eqref{eq:A2_long_w},
\begin{align}\label{eq:A2_result_w}
A_2&=\dfrac{b^2}{(4a-2)^2}E[a_3+2a_4]\nonumber\\
&=\dfrac{b^2}{(4a-2)^2}\left[\dfrac{24a}{b^2}+\dfrac{16a^2}{b^2}-\dfrac{24a}{b^2}+o\left(\dfrac{a}{b^2}\right)\right]\nonumber\\
&=\dfrac{16a^2}{(4a-2)^2}+o\left(\dfrac{1}{a}\right).
\end{align}
To calculate $A_3$ at \eqref{A3_w}, consider $p=0,...(a-1)$. For $q=2p$ and $q=2p+1,$ let $c=1/2$ in \eqref{eq:z1z2_0}
\begin{align*}
OL&=E[(\bar{B}_p-\bar{B})^2(\bar{B}_q(b/2)-\bar{B})^2]\\
&=\dfrac{4}{b^2}+\dfrac{3}{n^2}-\dfrac{7}{bn}.
\end{align*}
For $q\neq 2p$ and $q\neq 2p+1$, let $c=1/2$ in \eqref{eq:z1z2_1}. 
\begin{align*}
NOL &=E[(\bar{B}_p-\bar{B})^2(\bar{B}_q(b/2)-\bar{B})^2]\\
&=\dfrac{2}{b^2}+\dfrac{3}{n^2}-\dfrac{3}{bn}.
\end{align*}
Then $A_3$ at \eqref{A3_w} equals to 
\begin{align}\label{eq:A3_result_w}
A_3 &=\dfrac{-2b^2}{(a-1)(2a-1)}a\cdot[2OL+(2a-2)NOL)]\nonumber\\
    &=\dfrac{-4ab^2}{(a-1)(2a-1)}\left[\left(\dfrac{4}{b^2}+\dfrac{3}{n^2}-\dfrac{7}{bn}\right)+(a-1)\left(\dfrac{3}{n^2}+\dfrac{2}{b^2}-\dfrac{3}{bn}\right)\right]\nonumber\\
    &=\dfrac{-4ab^2}{(a-1)(2a-1)}\left[\dfrac{2a}{b^2}-\dfrac{1}{b^2}+o\left(\dfrac{1}{b^2}\right)\right]\nonumber\\
    &=\dfrac{4a}{1-a}+o\left(\dfrac{1}{a}\right).
\end{align}
Combine $A_1\ A_2\ A_3$ at \eqref{eq:A1_result_w},\eqref{eq:A2_result_w} and \eqref{eq:A3_result_w}, we can approach $E[\tilde{\sigma}^4_{w}]$ at \eqref{eq:A1+A2+A3_w} by
\begin{align*}
E[\tilde{\sigma}^4_{w}]&=A_1+A_2+A_3\\
&=\dfrac{4a^2}{(a-1)^2}+\dfrac{16a^2}{(4a-2)^2}-\dfrac{4a}{a-1}+o\left(\dfrac{1}{a}\right)\\
&=\dfrac{4a^4+8a^3-12a^2+4a}{(a-1)^2(2a-1)^2}+o\left(\dfrac{1}{a}\right).
\end{align*}
Next consider 
\begin{align*}
E[\tilde{\sigma}^2_{w}]&=E\left[\dfrac{2b}{a-1}\sum_{l=0}^{a-1}(\bar{B}_l^2-\bar{B})^2-\dfrac{b}{4a-2}\sum_{l=0}^{2a-1}(\bar{B}_l(b/2)-\bar{B})^2\right]\\
&=a\cdot\dfrac{2b}{a-1}\dfrac{n-b}{bn}-2a\cdot\dfrac{b}{4a-2}\dfrac{2n-b}{bn}=1.
\end{align*}
Therefore
\begin{align*}
\Var[\tilde{\sigma}^2_{w}]&=E[\tilde{\sigma}^4_{w}]-(E[\tilde{\sigma}^2_{w}])^2\\
&=\dfrac{4a^4+8a^3-12a^2+4a-(a-1)^2(2a-1)^2}{(a-1)^2(2a-1)^2}\\
&=\dfrac{20a^3+o(a^3)}{4a^4}\\
&=\dfrac{5b}{n}+o\left(\dfrac{b}{n}\right).
\end{align*}
\end{proof}

\begin{lemma}\label{lemma:Esq0}
Let $\hat{\sigma}^2_n$ be either $\hat{\sigma}^2_w$ or $\hat{\sigma}^2_f$. Assume Assumption \ref{ass:batch} holds. Further suppose Assumption \ref{ass:sip} holds with $ED^4<\infty$,  $E_Fg^4<\infty$ and as $n\rightarrow\infty$,
\begin{equation}\label{eq:th2}
\psi^2(n)b^{-1} \log n \rightarrow 0.
\end{equation}
Then 
\begin{equation}\label{eq:Esq0}
E[(\hat{\sigma}^2_n-\sigma^2\tilde{\sigma}^2_n)^2]\rightarrow\ 0\  \text{as}\ n\rightarrow\infty.
\end{equation}
\end{lemma}

\begin{proof}
Lemma B.4 of \cite{jone:hara:caff:neat:2006} shows 
\begin{equation}\label{eq:abs0}
|\hat{\sigma}^2_{bm}-\sigma^2\tilde{\sigma}^2_{bm}|\rightarrow 0\ \text{a.s}\ \text{as}\ n\rightarrow\infty
\end{equation}
under a slightly different condition from \eqref{eq:th2} and geometric ergodicity. Following the same argument, it can be shown that each component of $|\hat{\sigma}^2_{bm}-\sigma^2\tilde{\sigma}^2_{bm}|$ goes to 0 under Assumption~\ref{ass:sip} and \eqref{eq:th2}. Hence \eqref{eq:abs0} also holds under conditions of Lemma \ref{lemma:Esq0}. It then follows from Lemmas 12, 13 and 14 of \cite{fleg:jone:2010} that 
\begin{equation}\label{eq:Esq0bm}
E[(\hat{\sigma}^2_{bm}-\sigma^2\tilde{\sigma}^2_{bm})^2]\rightarrow\ 0\  \text{as}\ n\rightarrow\infty,
\end{equation}
and
\begin{equation}\label{eq:Esq0obm}
E[(\hat{\sigma}^2_{obm}-\sigma^2\tilde{\sigma}^2_{obm})^2]\rightarrow\ 0\  \text{as}\ n\rightarrow\infty.
\end{equation}
From \eqref{eq:ft=bt}, $\hat{\sigma}^2_{f}$ can be expressed by a linear combination of two Bartlett SV estimators which are asymptotically equivalent to $\hat{\sigma}^2_{obm}$, and $\hat{\sigma}^2_{w}$ can be expressed as a linear combination of two BM estimators, \eqref{eq:Esq0bm} and \eqref{eq:Esq0obm} result in Lemma \ref{lemma:Esq0}.
\end{proof}

\section{Proof of Theorem~\ref{thm:var}} \label{proof:var}
Corollary~\ref{cor:lemmaf} and Lemma \ref{lemma:w} show
\[\text{Var}[\tilde{\sigma}^{2}_{f}]=\dfrac{8b}{3n}+o\left(\dfrac{b}{n}\right),\]
and
\[\text{Var}[\tilde{\sigma}^{2}_{w}]=\dfrac{5b}{n}+o\left(\dfrac{b}{n}\right).\]
To derive $\Var[\hat{\sigma}^2_w]/\Var[\hat{\sigma}^2_f]$, we use Lemma~\ref{lemma:Esq0} to show
\[\dfrac{n}{b}\text{Var}[\hat{\sigma}^2_{f}]=\dfrac{8}{3}\sigma^4+o(1),\]
and
\[\dfrac{n}{b}\text{Var}[\hat{\sigma}^2_{w}]=5\sigma^4+o(1).\]

We only show results for $\Var[\hat{\sigma}^2_w]$ as $\Var[\hat{\sigma}^2_f]$ follows a similar proof. Define
\[
\eta=Var[\hat{\sigma}^{2}_{w}-\sigma^{2}\tilde{\sigma}^{2}_{w}]+2\sigma^{2}\cdot E[(\hat{\sigma}^{2}_{w}-\sigma^{2}\tilde{\sigma}^{2}_{w})(\tilde{\sigma}^{2}_{w}-E\tilde{\sigma}^{2}_{w})].
\]
As in \cite{fleg:jone:2010}, we show that $\eta\rightarrow 0$ as $n\rightarrow\infty$ using Cauchy-Schwarz inequality, $Var[X]\leq EX^{2}$ and \eqref{eq:Esq0}.  Specifically, we have
\begin{align*}
|\eta|&=|Var[\hat{\sigma}^{2}_{w}-\sigma^{2}\tilde{\sigma}^{2}_{w}]+2\sigma^{2}\cdot E[(\hat{\sigma}^{2}_{w}-\sigma^{2}\tilde{\sigma}^{2}_{w})(\tilde{\sigma}^{2}_{w}-E\tilde{\sigma}^{2}_{w})]|\\
&\leq E[(\hat{\sigma}^{2}_{w}-\sigma^{2}\tilde{\sigma}^{2}_{w})^{2}]+2\sigma^{2}\cdot\sqrt{E[(\hat{\sigma}^{2}_{w}-\sigma^{2}\tilde{\sigma}^{2}_{w})^{2}]\cdot E[(\tilde{\sigma}^{2}_{w}-E\tilde{\sigma}^{2}_{w})^{2}]}\\
&=E[(\hat{\sigma}^{2}_{w}-\sigma^{2}\tilde{\sigma}^{2}_{w})^{2}]+2\sigma^{2}\cdot (E[(\hat{\sigma}^{2}_{w}-\sigma^{2}\tilde{\sigma}^{2}_{w})^{2}])^{1/2}\cdot (E[(\tilde{\sigma}^{2}_{w}-E\tilde{\sigma}^{2}_{w})^{2}])^{1/2}\\
&=E[(\hat{\sigma}^{2}_{w}-\sigma^{2}\tilde{\sigma}^{2}_{w})^{2}]+2\sigma^{2}\cdot (E[(\hat{\sigma}^{2}_{w}-\sigma^{2}\tilde{\sigma}^{2}_{w})^{2}])^{1/2}\cdot (\Var[\tilde{\sigma}^{2}_{w}])^{1/2} .
\end{align*}
Since $E[(\hat{\sigma}^{2}_{w}-\sigma^{2}\tilde{\sigma}^{2}_{w})^{2}]=o(1)$ from Lemma~\ref{lemma:Esq0}, 
\begin{align*}
|\eta|&\le E[(\hat{\sigma}^{2}_{w}-\sigma^{2}\tilde{\sigma}^{2}_{w})^{2}]+2\sigma^{2}\cdot (E[(\hat{\sigma}^{2}_{w}-\sigma^{2}\tilde{\sigma}^{2}_{w})^{2}])^{1/2}\cdot (\Var[\tilde{\sigma}^{2}_{w}])^{1/2}  \\
&=o(1)+2\sigma^{2}[o(1)]^{1/2}\left(\dfrac{b}{n}\right)^{1/2}(5+o(1))^{1/2}\\
&=o(1)+2\sigma^{2}\left(\dfrac{b}{n}\right)^{1/2}(o(1)(5+o(1)))^{1/2}\\
&=o(1) ,
\end{align*}
since $b/n\rightarrow 0$ as $n\rightarrow\infty$. Then
\begin{align*}
Var[\hat{\sigma}^{2}_{w}]&=E[\hat{\sigma}^{2}_{w}-E\hat{\sigma}^{2}_{w}]^{2}\notag\\
&=E[\hat{\sigma}^{2}_{w}-\sigma^{2}\tilde{\sigma}^{2}_{w}+\sigma^{2}\tilde{\sigma}^{2}_{w}-\sigma^{2}E\tilde{\sigma}^{2}_{w}-E\hat{\sigma}^{2}_{w}+\sigma^{2}E\tilde{\sigma}^{2}_{w}]^{2}\notag\\
&=E[(\hat{\sigma}^{2}_{w}-\sigma^{2}\tilde{\sigma}^{2}_{w})+\sigma^{2}(\tilde{\sigma}^{2}_{w}-E\tilde{\sigma}^{2}_{w})-(E\hat{\sigma}^{2}_{w}-\sigma^{2}E\tilde{\sigma}^{2}_{w})]^{2}\notag\\
&=E[((\hat{\sigma}^{2}_{w}-\sigma^{2}\tilde{\sigma}^{2}_{w})-E[\hat{\sigma}^{2}_{w}-\sigma^{2}\tilde{\sigma}^{2}_{w}])+\sigma^{2}(\tilde{\sigma}^{2}_{w}-E\tilde{\sigma}^{2}_{w})]^{2}\notag\\
&=E[(\hat{\sigma}^{2}_{w}-\sigma^{2}\tilde{\sigma}^{2}_{w})-E[\hat{\sigma}^{2}_{w}-\sigma^{2}\tilde{\sigma}^{2}_{w}]]^{2}+E[\sigma^{2}(\tilde{\sigma}^{2}_{w}-E\tilde{\sigma}^{2}_{w})]^{2}\notag\\
&\ \ \ \ +2\sigma^{2}\cdot E[((\hat{\sigma}^{2}_{w}-\sigma^{2}\tilde{\sigma}^{2}_{w})-E[\hat{\sigma}^{2}_{w}-\sigma^{2}\tilde{\sigma}^{2}_{w}])\cdot(\tilde{\sigma}^{2}_{w}-E\tilde{\sigma}^{2}_{w})]\notag\\
&=E[(\hat{\sigma}^{2}_{w}-\sigma^{2}\tilde{\sigma}^{2}_{w})-E[\hat{\sigma}^{2}_{w}-\sigma^{2}\tilde{\sigma}^{2}_{w}]]^{2}+2\sigma^{2}\cdot E[(\hat{\sigma}^{2}_{w}-\sigma^{2}\tilde{\sigma}^{2}_{w})\cdot (\tilde{\sigma}^{2}_{w}-E\tilde{\sigma}^{2}_{w})]\notag\\
&\ \ \ \ +\sigma^{4}Var(\tilde{\sigma}^{2}_{w})-2\sigma^{2}E[(E[\hat{\sigma}^{2}_{w}-\sigma^{2}\tilde{\sigma}^{2}_{w}])(\tilde{\sigma}^{2}_{w}-E\tilde{\sigma}^{2}_{w})]\notag\\
&=\eta+\sigma^{4}Var(\tilde{\sigma}^{2}_w)\notag\\
&=5\sigma^{4}\cdot\dfrac{b}{n}+o\left(\dfrac{b}{n}\right).
\end{align*}
Similarly we can show
\[\Var[\hat{\sigma}^2_{f}]=\dfrac{8}{3}\sigma^4\cdot\dfrac{b}{n}+o\left(\dfrac{b}{n}\right).\]
Therefore 
\[\Var[\hat{\sigma}^2_w]/\Var[\hat{\sigma}^2_f]=1.875.\]

\section{Tukey-Hanning calculation}\label{app:TH}
To show the Tukey-Hanning window at \eqref{eq:th} satisfies \eqref{eq:cond1}, first consider
\begin{align*}
&\sum_{k=1}^{b-2}k\Delta_2w_{n}(k)=\sum_{k=1}^{b-2}\dfrac{k}{2}\left[\cos\left(\dfrac{\pi(k-1)}{b}\right)+\cos\left(\dfrac{\pi(k+1)}{b}\right)-2\cos\left(\dfrac{\pi k}{b}\right)\right]\\
&=\dfrac{1}{2}\left[\sum_{k=0}^{b-3}(k+1)\cos\left(\dfrac{\pi k}{b}\right)+\sum_{k=2}^{b-1}(k-1)\cos\left(\dfrac{\pi k}{b}\right)-2\sum_{k=1}^{b-2}k \cos\left(\dfrac{\pi k}{b}\right)\right]\\
&=\dfrac{1}{2}\left[\sum_{k=0}^{b-3}(k+1)\cos\left(\dfrac{\pi k}{b}\right)-\sum_{k=1}^{b-2}k \cos\left(\dfrac{\pi k}{b}\right)\right]\\
&\ +\dfrac{1}{2}\left[\sum_{k=2}^{b-1}(k-1)\cos\left(\dfrac{\pi k}{b}\right)-\sum_{k=1}^{b-2}k \cos\left(\dfrac{\pi k}{b}\right)\right]\\
&=\dfrac{1}{2}\left[1+\sum_{k=1}^{b-3}(k+1)\cos\left(\dfrac{\pi k}{b}\right)-\sum_{k=1}^{b-3}k \cos\left(\dfrac{\pi k}{b}\right)-(b-2)\cos\left(\dfrac{\pi(b-2)}{b}\right)\right]\\
&\ +\dfrac{1}{2}\left[(b-2)\cos\left(\dfrac{\pi(b-1)}{b}\right)+\sum_{k=2}^{b-2}(k-1)\cos\left(\dfrac{\pi k}{b}\right)-\cos\left(\dfrac{\pi}{b}\right)-\sum_{k=2}^{b-2}k\cos\left(\dfrac{\pi k}{b}\right)\right]\\
&=\dfrac{1}{2}\left[1+\sum_{k=1}^{b-3}\cos\left(\dfrac{\pi k}{b}\right)-(b-2)\cos\dfrac{\pi(b-2)}{b}\right]\\
&\ +\dfrac{1}{2}\left[(b-2)\cos\left(\dfrac{\pi(b-1)}{b}\right)-\cos\left(\dfrac{\pi}{b}\right)-\sum_{k=2}^{b-2}\cos\left(\dfrac{\pi k}{b}\right)\right]\\
&=\dfrac{1}{2}\left[1+\cos\left(\dfrac{\pi}{b}\right)-(b-2)\cos\left(\dfrac{\pi(b-2)}{b}\right)\right]\\
&\ +\dfrac{1}{2}\left[(b-2)\cos\left(\dfrac{\pi(b-1)}{b}\right)-\cos\left(\dfrac{\pi}{b}\right)-\cos\left(\dfrac{\pi(b-2)}{b}\right)\right]\\
&=\dfrac{1}{2}\left[1-(b-1)\cos\left(\dfrac{\pi(b-2)}{b}\right)+(b-2)\cos\left(\dfrac{\pi(b-1)}{b}\right)\right].
\end{align*}
Therefore
\begin{align*}
\sum_{k=1}^{b}k\Delta_2w_{n}(k)&=(b-1)\Delta_2 w_{th}(b-1)+b\Delta_2 w_{th}(b)+\sum_{k=1}^{b-2}k\Delta_2w_{th}(k)\\
&=(b-1)\left[\dfrac{1}{2}\cos\left(\dfrac{\pi(b-2)}{b}\right)-\cos\left(\dfrac{\pi(b-1)}{b}\right)-\dfrac{1}{2}\right]\\
&\ \ \ +b\left[\dfrac{1}{2}+\dfrac{1}{2}\cos\left(\dfrac{\pi(b-1)}{b}\right)\right]+\sum_{k=1}^{b-2}k\Delta_2w_{th}(k)\\
&=\left[\dfrac{b-1}{2}\cos\left(\dfrac{\pi(b-2)}{b}\right)+\dfrac{1}{2}-\dfrac{b-2}{2}\cos\left(\dfrac{\pi(b-1)}{b}\right)\right]\\
&\ \ \ +\left[\dfrac{1}{2}-\dfrac{b-1}{2}\cos\left(\dfrac{\pi(b-2)}{b}\right)+\dfrac{b-2}{2}\cos\left(\dfrac{\pi(b-1)}{b}\right)\right]\\
&=1.
\end{align*}

\end{appendix}

\bibliographystyle{apalike}
\bibliography{ref}
\end{document}